\documentclass{amsart}

\usepackage{amsthm, amssymb, amsmath, thmtools, thm-restate, url}
\usepackage{amsfonts}
\usepackage{ mathrsfs }
\usepackage{tikz, tikz-cd}
\usetikzlibrary{matrix}

\title{0-concordance of knotted surfaces and alexander ideals}
\author{Jason Joseph}
\address{Department of Mathematics, University of Georgia, Athens Georgia 30602}
\email{jjoseph@math.uga.edu}
\date{\today}

\declaretheorem[numberwithin=section]{theorem}
\declaretheorem[sibling=theorem]{corollary}
\declaretheorem[sibling=theorem]{proposition}
\declaretheorem[sibling=theorem]{lemma}
\declaretheorem[sibling=theorem]{question}
\declaretheorem[sibling=theorem]{conjecture}
\declaretheorem[style=definition,sibling=theorem]{definition}
\declaretheorem[style=remark,sibling=theorem]{remark}
\declaretheorem[style=remark,sibling=theorem]{example}

\theoremstyle{plain}

\newtheorem*{kl}{Key Lemma}

\theoremstyle{definition}

\theoremstyle{remark}

\newcommand{\bigzero}{\mbox{\normalfont\Large\bfseries 0}}

\begin{document}

\maketitle
\begin{abstract}
In this paper we provide a new obstruction to 0-concordance of knotted surfaces in $S^4$ in terms of Alexander ideals. We use this to prove the existence of infinitely many linearly independent 0-concordance classes and to provide the first proof that the submonoid of 2-knots is not a group. The main result is that the Alexander ideal induces a homomorphism from the 0-concordance monoid $\mathscr{C}_0$ of oriented surface knots in $S^4$ to the ideal class monoid of $\mathbb{Z}[t^{\pm1}]$. Consequently, any surface knot with nonprincipal Alexander ideal is not 0-slice and in fact, not invertible in $\mathscr{C}_0$. Many examples are given. We also characterize which ideals are the ideals of surface knots, generalizing a theorem of Kinoshita, and generalize the knot determinant to the case of nonprincipal ideals. Lastly, we show that under a mild condition on the knot group, the peripheral subgroup of a knotted surface is also a 0-concordance invariant.
\end{abstract}

\begin{section}{Introduction}

It is well known that the first elementary ideal of the Alexander module of a knotted surface, called the {\it Alexander ideal}, may not be principal. The first recorded instance of this was Example 12 of {\it A Quick Trip Through Knot Theory} by Ralph Fox \cite{foxqt}; later this 2-knot would be identified with the 2-twist spun trefoil of Zeeman. Hence many authors define the Alexander polynomial of a surface knot to be a generator of the smallest principal ideal which contains the Alexander ideal. One goal of this paper is to promote the study of the Alexander ideal as is. Indeed, the ideal class monoid of a PID is trivial, so the main result of this paper would be completely missed if one only considered principal ideals. The impetus for studying these comes from the ribbon obstruction for 2-knots: any ribbon 2-knot has a Wirtinger presentation of deficiency 1, and all such knot groups have principal Alexander ideals. What we show in this paper is that having a nonprincipal ideal is in fact a 0-sliceness obstruction, which unlike the ribbon obstruction generalizes to any genus.

A {\it 2-knot} in $S^4$ is a smooth embedding of a 2-sphere into the 4-sphere, considered up to isotopy. Kervaire proved that all 2-knots in $S^4$ are slice (concordant to the unknot) \cite{kervaire}, so it is natural to seek restricted forms of concordance. Paul Melvin introduced the notion of $n$-concordance in his thesis \cite{melvin}. Two 2-knots are $n$-concordant if they are joined by a concordance such that each component of a regular level set has genus at most $n$. By Kervaire's theorem, any two 2-knots are $n$-concordant for some $n$, but obstructing $n$-concordance has proven difficult. Melvin proved that 0-concordant 2-knots have diffeomorphic Gluck twists in 1977, but it was unknown until recently if any 2-knots in $S^4$ were not 0-slice (0-concordant to the unknot). This is Problem 1.105a on the Kirby problem list \cite{kirby}.

Sunukjian showed in \cite{sunukjianconc} that all genus $g$ surface knots in $S^4$ are concordant. He also extended the notion of 0-concordance to higher genus surfaces: genus $g$ surface knots $K$ and $J$ are 0-concordant if they are joined by a concordance in which each regular level set consists of a genus $g$ surface and possibly some genus 0 components. The set of all oriented surface knots in $S^4$ modulo 0-concordance forms a commutative monoid under connected sum, which we denote $\mathscr{C}_0$. The set of all 2-knots modulo 0-concordance is an important submonoid of $\mathscr{C}_0$, which we denote $\mathscr{K}_0$.

We produce a 0-concordance obstruction by determining exactly how the Alexander ideal can change during such a concordance. Namely, any surface knot which is 0-slice must have a principal ideal. This is analogous to the Fox-Milnor theorem, which says that if a classical knot is slice, its Alexander polynomial must factor as $f(t)f(t^{-1})$. 0-concordance is analogous to classical knot concordance in another fundamental way, in that both are the smallest equivalence relation generated by ribbon concordance. Consequently, the set of ribbon 2-knots is clearly contained in the 0-concordance class of the unknot, but the converse is not clear, hence the analogous {\it 0-slice ribbon problem}: is every 0-slice 2-knot ribbon? 

More generally, the Alexander ideals of 0-concordant surface knots are equivalent in the ideal class monoid of $\mathbb{Z}[t^{\pm1}]$. This comes from a useful factoring of 0-concordances into two opposing ribbon concordances, and the key lemma of this paper which shows that during a ribbon concordance the Alexander ideal changes by multiplication by a principal ideal. The ideal class monoid of an integral domain $R$, denoted $\mathcal{I}(R)$, is a quotient of the monoid of nonzero ideals of $R$ under multiplication, by the equivalence relation $I\sim J$ if there exist nonzero $x,y\in R$ such that $(x)I=(y)J$. The Alexander ideal of a connected sum is the product of the individual ideals, so the operations on these monoids are compatible.

\begin{restatable}{theorem}{homomorphism}
\label{thm:hom}
The Alexander ideal induces a homomorphism $\Delta:\mathscr{C}_0\rightarrow \mathcal{I}(\mathbb{Z}[t^{\pm1}])$.
\end{restatable}

Of course 0-concordant surface knots must have the same genus; still we say a surface knot $K$ is {\it 0-slice} if it is 0-concordant to the unknotted surface of the same genus, and {\it invertible} if there exists a surface knot $J$ so that $K\# J$ is 0-slice (for a 2-knot, having no inverse in $\mathscr{C}_0$ is stronger than having no inverse in $\mathscr{K}_0)$. Since the Picard group of $\mathbb{Z}[t^{\pm1}]$ is trivial, no nontrivial ideal class is invertible.

\begin{restatable}{corollary}{noninvertible}
\label{cor:noninv}
If a surface knot $K$ has nonprincipal Alexander ideal, then it has no inverse in $\mathscr{C}_0$, i.e.\ for all surface knots $J$, $K\# J$ is not 0-slice.
\end{restatable}

An analysis of the effect of twist spinning on the Alexander ideal shows that many twist spins of the same knot have nonprincipal ideals, as long as the knot's determinant is not a unit. The $n$-twist spin of $K$ is denoted $\tau^n K$.

\begin{restatable}{theorem}{inftytwisty}\label{thm:inftytwisty}
If $K$ is a classical knot such that $|\Delta_K(-1)|\neq1$, then there exist infinitely many $n\in\mathbb{Z}$ such that $\Delta(\tau^{n} K)$ is not principal. In particular, if $n$ is even and $\Delta(\tau^{n} K)$ is principal, then $\Delta_K(t)$ has a root $z$ such that $z^{n}=1$.
\end{restatable}

In the special case of 2-twist spins of 2-bridge knots, the determinant is actually an invariant of 0-concordance.

\begin{restatable}{corollary}{determinant}\label{cor:2twist}
\label{cor:det}

Let $K$ and $J$ be 2-bridge knots. If the determinant $|\Delta_K(-1)|$ of K is not 1, then $\tau^2 K$ is not invertible in $\mathscr{C}_0$. 
If $\tau^2 K$ and $\tau^2 J$ are 0-concordant, then their determinants are the same.

\end{restatable}

The importance of Corollary \ref{cor:2twist} is mainly that it allows us to obstruct 0-concordance between some 2-knots which do bound rational homology spheres that are spin rational homology cobordant. See Remark \ref{rem:comparison} for a discussion.

The structure of the ideal class monoid is in general quite complicated, although in the case of ideals which admit a prime factorization we understand the situation completely. In particular, the maximal ideals of $\mathbb{Z}[t^{\pm1}]$ independently generate a free commutative submonoid of the ideal class monoid. If $K$ is a 2-bridge knot with prime determinant, then the Alexander ideal of its 2-twist spin is maximal. Hence any such collection with pairwise distinct determinants yields a basis for a free commutative submonoid of $\mathscr{K}_0$.

\begin{restatable}{theorem}{bigmonoid}
\label{thm:infbasis}
$\mathscr{K}_0$ contains a submonoid isomorphic to $\mathbb{N}^\infty$.
\end{restatable}

The techniques of this paper provide a very different answer to the 0-concordance problem than was provided recently by several authors \cite{sunukjian0conc}, \cite{daimiller}. However, all of these approaches utilize the factoring of a 0-concordance into two ribbon concordances. This was first made explicit by Sunukjian in \cite{sunukjian0conc}, and is also essentially pointed out in \cite{kirby}.

Sunukjian used this and techniques from Heegaard Floer homology to prove that if 2-knots $K$ and $J$ are 0-concordant and bound rational homology sphere Seifert manifolds $M^\circ$ and $N^\circ$, then the $d$-invariants of $M$ and $N$ must coincide \cite{sunukjian0conc}. This established the existence of infinitely many 0-concordance classes of 2-knots. Dai and Miller proved that $M$ and $N$ as above must be spin rational homology cobordant, and utilized linear independence in the spin rational homology cobordism group to obtain infinitely many linearly independent 0-concordance classes of 2-knots \cite{daimiller}.

Next, we classify the ideals which occur as Alexander ideals of knotted surfaces. Kinoshita proved in 1960 that any polynomial $f(t)\in\mathbb{Z}[t^{\pm1}]$ satisfying $f(1)=\pm 1$ is the Alexander polynomial of a ribbon 2-knot. For an ideal $I\subseteq \mathbb{Z}[t^{\pm1}]$ and $a\in\mathbb{Z}$, let $I|_{t=a}$ denote the nonnegative generator of $\{f(a): f(t)\in I\}\subseteq\mathbb{Z}$. We prove the following generalization of Kinoshita's theorem to arbitrary genus.

\begin{restatable}{theorem}{ideal}
\label{thm:ideal}
An ideal $I$ of $\mathbb{Z}[t^{\pm1}]$ is the Alexander ideal of a surface knot if and only if $I|_{t=1}=1$.
\end{restatable}

This characterization is constructive: an ideal with the above property which is minimally generated by $g$ elements is the ideal of a ribbon surface knot of genus $g$.

Theorem \ref{thm:ideal} allows us to describe the image of $\Delta$ explicitly. Let $\mathcal{I}_K$ be the submonoid $\{[I]:I|_{t=1}=1\}$ of $\mathcal{I}(\mathbb{Z}[t^{\pm1}])$.

\begin{restatable}{corollary}{image}
\label{thm:image}
The image of $\Delta$ is $\mathcal{I}_K$.
\end{restatable}

\begin{restatable}{corollary}{nonreversible}\label{cor:nonreversible}
There exist infinitely many ribbon tori in $S^4$ which are not 0-concordant to their reverses.
\end{restatable}

Next we generalize the notion of knot determinant to nonprincipal ideals and prove some propositions. Namely, the determinant of the $n$-twist spin of $K$ is the same as the determinant of $K$ if $n$ is even, and 1 if $n$ is odd. Likewise, a nontrivial Fox $p$-coloring of $K$ extends to a nontrivial coloring of $\tau^n K$ if and only if $n$ is even.

Lastly, we show that the peripheral subgroup $P(K)$ is a 0-concordance invariant for reasonable knot groups. Let $X=\overline{S^4\setminus\nu K}$ and $i:\partial X\hookrightarrow X$ the inclusion. The peripheral subgroup $P(K)$ is $i_*(\pi_1\partial X)\leq\pi_1 X=\pi K$.

\begin{restatable}{theorem}{peripheral}
If $K_0$ is 0-concordant to $K_1$ and the knot groups of $K_0$ and $K_1$ are residually finite or locally indicable, then $P(K_0)\cong P(K_1)$.
\end{restatable}
This leads to another source of examples for surface knots with genus at least one. We show that ideal classes and the peripheral subgroup are independent invariants of 0-concordance, in that either can be trivial while the other is not. We also show that there are many independent 0-concordance classes represented by ribbon tori, in contrast with the genus 0 case, where all ribbon 2-knots are 0-slice.

\begin{subsection}*{Organization}\
In Section 2 we give a summary of the essential tools from Fox calculus for calculating elementary ideals from a presentation of the knot group. Then we define the various notions of concordance which will be useful, and make an in-depth study of ribbon concordance. In Section 3 we develop some results about the ideal class monoid of $\mathbb{Z}[t^{\pm1}]$. In Section 4 we prove the main theorem and applications. In Section 5 we characterize Alexander ideals of surface knots and generalize the determinant to the case of nonprincipal ideals. In Section 6 we study 0-concordance of higher genus surfaces from the perspective of the peripheral subgroup, and in Section 7 we make some observations and questions.

\end{subsection}

\begin{subsection}*{Acknowledgments}\

The author would like to thank Pete L. Clark, David Gay, Danny Krashen, Jeffrey Meier, Maggie Miller, Patrick Naylor, and Nathan Sunukjian for helpful comments and conversations, and the Max Planck Institute in Bonn, Germany, where some of this research was conducted.
\end{subsection}

\end{section}
\begin{section}{Preliminaries}
\begin{subsection}{Conventions}\

Throughout we will consider smooth, closed, connected, oriented surface knots $K:\Sigma_g\hookrightarrow S^4$. When $g=0$, $K$ is called a {\it 2-knot}. The {\it knot group} is $\pi K:=\pi_1(S^4\setminus K)$. 

If $S\subseteq M$ is a subset of a monoid $M$, we make the convention that the {\it submonoid of $M$ generated by $S$} is the smallest submonoid of $M$ containing $S$ and the identity of $M$. By $\mathbb{N}$ we denote the monoid of nonnegative integers under addition.
\end{subsection}

\begin{subsection}{Fox Calculus} \label{sec:fox}\

Since several of our proofs explicitly make use of Fox's free differential calculus, we give a brief description. For more details see \cite{crowfox}, \cite{foxqt}, \cite{foxcalci}. Any group homomorphism $G\rightarrow H$ has a unique extension to a ring homomorphism between the group rings $\mathbb{Z}G\rightarrow \mathbb{Z}H$. When $G$ is a knot group, expressed as a presentation $\langle x_1,\dots,x_n|r_1,\dots,r_m\rangle$, the two homomorphisms we will consider are the quotient map $F\rightarrow G$ which defines $G$ from a free group on $n$ generators, and the abelianization map $G\rightarrow \langle t\rangle\cong\mathbb{Z}$. The benefit of this last homomorphism is that all knot groups abelianize to $\mathbb{Z}\cong H_1(S^4\setminus \Sigma_g)$, so we have a well-defined universe in which to compare, and since $\mathbb{Z}$ is abelian, its group ring is commutative, so one can define determinants and elementary ideals.

 A {\it derivative} is a linear mapping $D:\mathbb{Z}F\rightarrow\mathbb{Z}F$ which obeys a Leibniz rule. On elements of $F$, this takes the form $D(g_1 g_2)=Dg_1 +g_1 Dg_2$, and then one extends linearly to define $D$ on all of $\mathbb{Z}F$. We are concerned with the case that $F$ is a free group, generated by $x_1,\dots,x_n$. Then for each free generator $x_j$ there is a unique derivative $\partial/\partial x_j$ satisfying $\partial x_i/\partial x_j=\delta_{ij}$. Note that $\partial 1/\partial x=0$ and $\partial x^{-1}/\partial x = -x^{-1}$.
 
 The {\it Alexander matrix} corresponding to a knot group presentation\\ $P=\langle x_1,\dots,x_n|r_1,\dots,r_m\rangle$ has entries the images of the $r_i$ under the composition
 
 $$\mathbb{Z}F\xrightarrow{\partial/\partial x_j}\mathbb{Z}F\xrightarrow{\text{    $\gamma$    }}\mathbb{Z}G\xrightarrow{\text{    $\mathfrak{a}$    }}\mathbb{Z}\langle t\rangle$$
 where $F\xrightarrow{\text{  $\gamma$  }} F/R\cong G$ is the canonical homomorphism defining $G$ from $P$ and $G\xrightarrow{\text{ $\mathfrak{a}$ }}\langle t\rangle$ is the abelianizer. So $A=(a_{ij})$, where $a_{ij}=\mathfrak{a}\gamma(\partial r_i/\partial x_j)$. It is a presentation matrix for the Alexander module. Two matrices are considered equivalent if one is obtained from the other by a sequence of row and column operations, adding a row of zeroes, or stabilization: 
 $A\rightarrow \begin{pmatrix}
A & \vec{0}\\
\vec{0} & 1
 \end{pmatrix}$. Different presentations of the same group give rise to equivalent Alexander matrices.
 
Starting with $P$ as before, $A$ will be an $(m\times n)$-matrix. The $k^\text{th}$ {\it elementary ideal} $\varepsilon_k$ of $A$ is the ideal of $\mathbb{Z}\langle t\rangle$ generated by the determinants of the square $(n-k)$-minors of $A$. When $n-k\leq0$, $\varepsilon_k=(1)$. The {\it Alexander ideal} is the first elementary ideal. Equivalent matrices define the same chain of elementary ideals, so these are invariants of the oriented knot $K$. Since $t$ encodes an orientation, these are really invariants of the pair $(\pi K,\varepsilon)$, where $\varepsilon$ is an orientation of $K$. 
 
$P$ is a {\it Wirtinger presentation} if all generators $x_i$ abelianize to $t$, the generator of $H_1(S^4\setminus K)$, and all relations are of the form $x_i=w x_j w^{-1}$, where $i,j\in\{1,\dots,n\}$ and $w$ is a word in the $x_i$. A nice consequence of all generators abelianizing to $t$ is that the sum across any row of the Alexander matrix is zero, so we can always replace one column with a column of zeroes when working with Wirtinger presentations (Theorem 8.3.7 \cite{crowfox}). When $K$ is a classical knot or a ribbon $n$-knot, it has a Wirtinger presentation with $m+1$ generators and $m$ relations. After replacing one column with zeroes, we see that there is only one $(m\times m)$-minor with a nonzero determinant, so these knots always have principal Alexander ideals. In this case a generator of the ideal is called the {\it Alexander polynomial of $K$}.

 \begin{example}
 To illustrate these techniques we use the 2-twist spun trefoil as an example, with its standard Wirtinger presentation $\langle x,y|xyxy^{-1}x^{-1}y^{-1},x^2yx^{-2}y^{-1}\rangle$.
 
 $$\begin{pmatrix}
 \frac{\partial r_1}{\partial x} & \frac{\partial r_1}{\partial x}\\
 \frac{\partial r_2}{\partial x} & \frac{\partial r_2}{\partial x}
 \end{pmatrix}=
 \begin{pmatrix}
 1+xy-xyxy^{-1}x^{-1} & x -xyxy^{-1}-xyxy^{-1}x^{-1}y^{-1}\\
 1+x-x^2yx^{-1}-x^2yx^{-2} & x^2-x^2yx^{-2}y^{-1}
 \end{pmatrix}$$
 
 $$\xrightarrow{  a\gamma  }
 \begin{pmatrix}
 1+t^2-t & t-t^2-1\\
 1+t-t^2-t & t^2-1
 \end{pmatrix}\sim
 \begin{pmatrix}
 t^2-t+1 & 0\\
 1-t^2 & 0
 \end{pmatrix}
 $$
 Since there are 2 columns, the Alexander ideal is generated by the $(1\times 1)$-minors, and $\Delta(K)=(t^2-t+1,t^2-1)$. One can check this is equal to $(3,t+1)$, which makes it clear that the quotient $\mathbb{Z}[t^{\pm1}]/(3,t+1)\cong\mathbb{Z}_3$, so $\Delta(K)$ is maximal. $\mathbb{Z}[t^{\pm1}]$ is a regular ring of Krull dimension 2, so every maximal ideal is minimally generated by 2 elements. Thus the 2-twist spun trefoil is not ribbon.
 \end{example}
 
 The following may be known, but we could not find a proof in the literature. This is of critical importance to Theorem \ref{thm:hom}, so we include a proof for the convenience of the reader.
 
\begin{restatable}{proposition}{additive}\label{prop:sum}
 Let $K$ and $J$ be surface knots. Then $\Delta(K\# J)=\Delta(K)\Delta(J)$. 
\end{restatable}

 \begin{proof}
 Let $\langle x_1,\dots,x_n|r_1,\dots,r_k\rangle$ and $\langle y_1,\dots, y_m|s_1,\dots,s_l\rangle$ be Wirtinger presentations for $\pi K$ and $\pi J$. Then $\langle x_1,\dots,x_n,y_1,\dots,y_m|x_1y_1^{-1},r_1,\dots,r_k,s_1\dots,s_l\rangle$ is a Wirtinger presentation for $K\# J$. Let $a_{ij}=\mathfrak{a}\gamma\left(\frac{\partial r_i}{\partial x_j}\right)$ and $b_{ij}=\mathfrak{a}\gamma\left(\frac{\partial s_i}{\partial y_j}\right)$. Then the Alexander matrix for $K$ is $A=(a_{ij})$ and the matrix for $J$ is $B=(b_{ij})$, so the matrix for $K\# J$ is:
 
$$\begin{pmatrix}

    1 & 0 & \dots & 0 & -1 & 0 & \dots & 0\\
    a_{11} & a_{12} & \dots & a_{1n} & 0 & 0 & \dots & 0\\
    \vdots & \vdots & \ddots & \vdots & \vdots & \vdots & \ddots & \vdots\\
    a_{k1} & a_{k2} & \dots & a_{kn} & 0 & 0 & \dots & 0\\
    0 & 0 & \dots & 0 & b_{11} & b_{12} & \dots & b_{1m}\\
    \vdots & \vdots & \ddots & \vdots & \vdots & \vdots & \ddots & \vdots\\
    0 & 0 & \dots & 0 & b_{l1} & b_{l2} & \dots & b_{lm}
    \end{pmatrix}
    $$
    
    As usual, we replace a column by zero, the sum of all columns. It is convenient to replace the $(n+1)^{\text{st}}$ column with zeroes, resulting in 
    $$\begin{pmatrix}

    1 & 0 & \dots & 0 & 0 & 0 & \dots & 0\\
    a_{11} & a_{12} & \dots & a_{1n} & 0 & 0 & \dots & 0\\
    \vdots & \vdots & \ddots & \vdots & \vdots & \vdots & \ddots & \vdots\\
    a_{k1} & a_{k2} & \dots & a_{kn} & 0 & 0 & \dots & 0\\
    0 & 0 & \dots & 0 & 0 & b_{12} & \dots & b_{1m}\\
    \vdots & \vdots & \ddots & \vdots & \vdots & \vdots & \ddots & \vdots\\
    0 & 0 & \dots & 0 & 0 & b_{l2} & \dots & b_{lm}
    \end{pmatrix}\sim
    \begin{pmatrix}
    1 & 0 & \dots & 0 & 0 & 0 & \dots & 0\\
    0 & a_{12} & \dots & a_{1n} & 0 & 0 & \dots & 0\\
    \vdots & \vdots & \ddots & \vdots & \vdots & \vdots & \ddots & \vdots\\
    0 & a_{k2} & \dots & a_{kn} & 0 & 0 & \dots & 0\\
    0 & 0 & \dots & 0 & 0 & b_{12} & \dots & b_{1m}\\
    \vdots & \vdots & \ddots & \vdots & \vdots & \vdots & \ddots & \vdots\\
    0 & 0 & \dots & 0 & 0 & b_{l2} & \dots & b_{lm}
    \end{pmatrix}$$
    
    $$\sim
    \begin{pmatrix}
    a_{12} & \dots & a_{1n} & 0 & 0 & \dots & 0\\
     \vdots & \ddots & \vdots & \vdots & \vdots & \ddots & \vdots\\
     a_{k2} & \dots & a_{kn} & 0 & 0 & \dots & 0\\
     0 & \dots & 0 & 0 & b_{12} & \dots & b_{1m}\\
    \vdots & \ddots & \vdots & \vdots & \vdots & \ddots & \vdots\\
     0 & \dots & 0 & 0 & b_{l2} & \dots & b_{lm}
    \end{pmatrix}$$
    
    In the last step we used the inverse of the stabilization move. To simplify things further, we delete the column of zeroes and remember to take the square minors which use all of the columns, i.e.\ the minors of size $n+m-2$.
    
    $$    \begin{pmatrix}
    a_{12} & \dots & a_{1n} & 0  & \dots & 0\\
     \vdots & \ddots & \vdots & \vdots & \ddots & \vdots\\
     a_{k2} & \dots & a_{kn} & 0 & \dots & 0\\
     0 & \dots & 0 &  b_{12} & \dots & b_{1m}\\
    \vdots & \ddots &  \vdots & \vdots & \ddots & \vdots\\
     0 & \dots & 0  & b_{l2} & \dots & b_{lm}
    \end{pmatrix}=
    \begin{pmatrix}
        A^\prime & \bigzero\\
        \bigzero & B^\prime
    \end{pmatrix}
    $$
    
    Note that $A^\prime$, $B^\prime$ are obtained from $A$, $B$ by deleting the first column, which may as well have been zero anyway. The claim now is that unless we choose a minor with $(n-1)$ rows from $A^\prime$ and $(m-1)$ rows from $B^\prime$, we will get a determinant of zero. Without loss of generality, suppose that we chose at least $n$ rows from $A^\prime$. This minor is of the form 
 $\begin{pmatrix}
     A^{\prime\prime} & \bigzero\\
     \bigstar & C
 \end{pmatrix}$, where both $A^{\prime\prime}$ and $C$ are square and the matrix $C$ has a row of zeroes. The determinant of this minor is $|A^{\prime\prime}|\cdot|C|=|A^{\prime\prime}|\cdot0=0$. Therefore, the only minors with nonzero determinants are of the form 
 $\begin{pmatrix}
     A^{\prime\prime} & \bigzero\\
     \bigzero & B^{\prime\prime}
 \end{pmatrix}$, where $A^{\prime\prime}$ and $B^{\prime\prime}$ are $(n-1)$ and $(m-1)$-minors of $A^\prime$ and $B^\prime$, respectively. The determinant is $|A^{\prime\prime}|\cdot|B^{\prime\prime}|$. Note that $|A^{\prime\prime}|$ is a generator of $\Delta(K)$ and $|B^{\prime\prime}|$ is a generator of $\Delta(J)$. By choosing all possible minors of this form, we obtain a generating set for $\Delta(K\# J)$, each generator equal to the product of a generator of $\Delta(K)$ and a generator of $\Delta(J)$. Since ranging through all $A^{\prime\prime}$ provides a generating set for $\Delta(K)$, and likewise with all $B^{\prime\prime}$ and $\Delta(J)$, $\Delta(K\# J)$ is equal to the product of ideals $\Delta(K)\Delta(J)$.
 \end{proof} 
\end{subsection}

\begin{subsection}{Twist spun knots}\label{sec:twistspin}\

Twist spinning was introduced by Zeeman in \cite{zeeman}. It is an operation which takes a classical knot in $S^3$ and produces a 2-knot in $S^4$, by twisting a knotted arc for $K$ an integer number of times while spinning it through the fourth dimension. Let $K$ be a classical knot and consider the $n$-twist spin of $K$, denoted $\tau^n K$. Zeeman proved that if the number of twists is at least 1, the resulting 2-knot is fibered by the $n$-fold cyclic branched cover of $K$ (hence $\tau^1 K=\mathcal{U}$ for any $K$). The knot group of $\tau^n K$ is obtained as a quotient of $\pi K$ by making the $n^\text{th}$ power of a meridian in the center of the group. Let $\langle x_0,\dots,x_m|r_1,\dots,r_m\rangle$ be a Wirtinger presentation for $\pi K$ (such a presentation can be obtained from any diagram for $K$). Then $\langle x_0,\dots,x_m|r_1,\dots,r_m,[x_0^n,x_1],\dots,[x_0^n,x_m]\rangle$ is a Wirtinger presentation for $\pi (\tau^n K)$. Sometimes it will be convenient to use the equivalent presentation $\langle x_0,\dots,x_m|r_1,\dots,r_m,x_0^nx_1^{-n},\dots,x_0^nx_m^{-n}\rangle$. These are equivalent because all the meridians of $\pi K$ are conjugate (any two elements which are conjugate and in the center of a group must be equal, and conversely if $x_0^n=x_i^n$, then $x_0^nx_i=x_i^{n+1}=x_ix_0^n$). In Theorem \ref{thm:inftytwisty} we work out the ideals of these 2-knots explicitly. 

The case of 2-bridge knots is especially simple and will come up several times, so we settle it now. If $K$ is 2-bridge then $\pi K$ has a Wirtinger presentation $\langle x,y|r\rangle$, and $\pi(\tau^n K)$ is then $\langle x,y|r,[x^n,y]\rangle\cong\langle x,y|r,x^ny^{-n}\rangle$. The Alexander ideal, as computed from these presentations, is $(\Delta_K(t),t^n-1)=(\Delta_K(t),\frac{t^n-1}{t-1})$, where $\Delta_K(t)$ is the Alexander polynomial of $K$.
\end{subsection}

\begin{subsection}{Concordance of surface knots}\

In this section we define the various notions of concordance which will be of interest. Let $K_0$ and $K_1$ be oriented surface knots of genus $g$ in $S^4$.

\begin{definition}
A {\bf concordance} between $K_0$ and $K_1$ is a smooth embedding\\
$C:\Sigma_g\times I\hookrightarrow S^4\times I$ such that $C|_{\Sigma_g\times\{i\}}=K_i$ for $i=0,1$, and such that projection onto the $I$ factor is Morse.
\end{definition}

\begin{definition}
A \textbf{ribbon concordance} $K_0\rightarrow K_1$ is a concordance $C$ with critical points of index 0 and 1 only.
\end{definition}

Note that ribbon concordance is not symmetric. The historical terminology for $K_0\rightarrow K_1$ is ``$K_1$ is ribbon concordant to $K_0$", denoted $K_1\geq K_0$. The arrow in our notation is to indicate the direction of time during the concordance. Also note that a 2-knot $K$ is \textit{ribbon} if and only if there is a ribbon concordance $\mathcal{U}\rightarrow K$.

\begin{definition}
A \textbf{0-concordance} between $K_0$ and $K_1$ is a concordance $C$ such that at each regular level set, $S^4_t\cap C$ consists of a connected genus $g$ surface and possibly some genus 0 components.
\end{definition}

So far all of the theorems which obstruct 0-concordance between surface knots utilize the following factorization of a 0-concordance into two opposing ribbon concordances. 

\begin{proposition}[Sunukjian] \label{prop:factor}
If $K_0$ and $K_1$ are 0-concordant surface knots, then there exists a surface knot $J$ and ribbon concordances $K_0\rightarrow J\leftarrow K_1$.
\end{proposition}

\begin{proof}
Let $C:\Sigma_g\times I\hookrightarrow S^4\times I$ be a 0-concordance. We can isotope $C$ ambiently so that all index 0 and 1 critical points occur before any index 2 or 3 critical points. So $C$ has a handle decomposition where we attach all 0-handles and 1-handles before any 2 or 3-handles. If a 1-handle was cancelled by a 2-handle, then its feet were attached to a single component of the level set in which it was attached, thereby increasing the genus of that component, hence $C$ is not a 0-concordance. Therefore, all 1-handles are cancelled by 0-handles, and since the concordance is connected there must be the same number of 0 and 1-handles. Turning the concordance upside down, the same must be true of the 2 and 3-handles, which form a ribbon concordance in the reversed direction.
\end{proof}

\end{subsection}

\begin{subsection}{Ribbon concordance}\

In \cite{gordonrc}, Gordon proved that for a ribbon concordance $C:S^1\times I\hookrightarrow S^3\times I$ of classical knots $K_i=C|_{S^1\times \{i\}}$, the knot groups of $K_0$, $C$, and $K_1$ obey\newline 
(i) $\pi K_0\hookrightarrow\pi C$ and (ii) $\pi K_1 \twoheadrightarrow \pi C$.

The following proposition displays the difference between ribbon concordance in the classical case with all higher dimensions, and suggests that the knot groups of $K_0$ and $K_1$ should play a fundamental role. Namely, the surjection above becomes an isomorphism, so by composing with its inverse there is a homomorphism from the group of $K_0$ to the group of $K_1$, which in many cases remains injective. Let $X_i=(S^4\times\{i\})\setminus \nu K_i$ and $Y=(S^4\times I) \setminus \nu C$ (so $\pi K_i=\pi_1 X_i$, $\pi C=\pi_1 Y$).

\begin{proposition} \label{prop:hom}
A ribbon concordance $K_0\rightarrow K_1$ induces a homomorphism \newline $\phi:\pi K_0\rightarrow \pi K_1$. 
\end{proposition}

\begin{proof}
We recall Gordon's proof; the only change is due to the dimension of the cobordism being one higher. Let $C$ be a ribbon concordance $K_0\rightarrow K_1$. Since the projection onto $I$ is Morse, $Y$ can be built from $X_0\times I$ by adding handles. Every time we pass a critical point of index 0, respectively 1, of $C$, we get a critical point of index 1, respectively 2, in $Y$. From this perspective
$$Y=(X_0\times I)\cup (\text{1-handles})\cup(\text{2-handles})$$
In order for $C$ to be a concordance, we must have added the same number of 0 and 1-handles. Therefore, $\pi C=\frac{\pi K_0*\langle z_1,\dots,z_n\rangle}{\langle\langle r_1, \dots , r_n\rangle\rangle}$, where each $r_i$ is of the form $z_i w_ix^{-1}w_i^{-1}$, for some meridian $x$ of $K_0$ (each $r_i$ can be chosen to be a Wirtinger relator).

Turning the cobordism upside down, we have 
$$Y=(X_1\times I)\cup (\text{3-handles})\cup(\text{4-handles})$$
Thus the inclusion $X_1\hookrightarrow Y$ induces an isomorphism on fundamental groups: $\pi K_1\cong \pi C$. The inclusion $X_0\hookrightarrow Y$ induces a homomorphism $\pi K_0\rightarrow\pi C$, so composing with the inverse of the isomorphism yields a homomorphism $\phi:\pi K_0\rightarrow \pi K_1$ induced by these inclusions.
\end{proof}

\begin{remark} \label{rem:inj} 
\normalfont It is conjectured that the homomorphism $\phi$ is always injective; this is a strong form of the Kervaire conjecture. Gordon's original proof of injectivity applies whenever $\pi K_0$ is residually finite, so in this case $\phi$ will be injective for any ribbon concordance $K_0\rightarrow K_1$. All 3-manifold groups are residually finite, so for classical ribbon concordance this is sufficient. Moreover, cyclic extensions of residually finite groups are residually finite, so the group of any fibered 2-knot is as well. Together these include all spun and twist spun knots. Gordon points out that $\pi K_0$ locally indicable is also sufficient. Thus the group map can obstruct some ribbon concordances; for instance it gives an easy proof of Corollary 2.1(i) of \cite{rcquandles}, which states that for $p,q$ distinct primes, there is no ribbon concordance $\tau^2T(2,p)\rightarrow \tau^2T(2,q)$ (the group of $\tau^2 T(2,p)$ is isomorphic to $\mathbb{Z}_p\rtimes\mathbb{Z})$. \end{remark}

\begin{remark}
The homomorphism $\phi$ sends meridians of $K_0$ to meridians of $K_1$, so in fact a ribbon concordance $K_0\rightarrow K_1$ induces a quandle homomorphism $\varphi:Q(K_0)\rightarrow Q(K_1)$, where $Q(K)$ is the fundamental quandle. This is equivalent to the diagrammatic interpretation in \cite{rcquandles}, where it is shown that a coloring of $K_1$, i.e.\ a quandle homomorphism $Q(K_1)\rightarrow X$, induces a coloring of $K_0$. The induced coloring is the composition $Q(K_0)\xrightarrow{\varphi}Q(K_1)\rightarrow X$.
\end{remark}

The proof of Proposition \ref{prop:hom} shows that, given any presentation for $\pi K_0$, we can obtain a presentation for $\pi K_1$ with the same number of generators and relations. The {\it deficiency} of a finitely presentable group $G$ is the maximal difference $g-r$ between the number of generators and relators, taken over all finite presentations $\langle x_1,\dots,x_g|s_1,\dots,s_r\rangle$ of $G$.

\begin{corollary}
If $K_0\rightarrow K_1$ is a ribbon concordance, then $\mathrm{def}(\pi K_0)\leq \mathrm{def}(\pi K_1)$.
\end{corollary}

We end this section with the key lemma for Theorem \ref{thm:hom}.
This is a generalization of the fact that ribbon 2-knots have principal Alexander ideals, because every ribbon 2-knot $K$ is the result of a ribbon concordance $\mathcal{U}\rightarrow K$.
\begin{kl} \label{lem:kl}
Let $K_0\rightarrow K_1$ be a ribbon concordance. Then $\Delta(K_1)=(f)\Delta(K_0)$, for some $f\in\mathbb{Z}[t^{\pm1}]$.
\end{kl}

\begin{proof}
By Proposition \ref{prop:hom}, a Wirtinger presentation for the knot group of $K_1$ can be obtained from a Wirtinger presentation for $\pi K_0$ by adding $m$ generators and relations, corresponding to the index 0 and 1 critical points of the concordance, respectively. So if $\pi K_0\cong\langle x_1,\dots,x_n|r_1,\dots,r_j\rangle=P_0$ is Wirtinger, then there is a Wirtinger presentation $P_1=\langle x_1,\dots,x_n,z_1,\dots,z_m|r_1,\dots,r_j,s_1,\dots,s_m\rangle$ for the knot group of $K_1$. From this presentation we compute the Alexander ideal of $K_1$ using Fox calculus. Below $A$ is the Jacobian corresponding to $P_0$, which gives rise to the matrix on the right hand side for $P_1$. $\Delta(K_1)$ is the ideal of $\mathbb{Z}[t^{\pm1}]$ generated by the determinants of all of the $(n+m-1)$-minors of the abelianized matrix. Since this is a Wirtinger presentation, we can replace the first column with a column of zeroes. This amounts to the realization that we may as well leave the first column out of any of our chosen minors, which leaves only the last $n+m-1$ columns.

$$ \begin{pmatrix}
    \frac{\partial r_1}{\partial x_1}  & \dots & \frac{\partial r_1}{\partial x_n}  \\
  \vdots & \ddots &  \vdots \\
   
   \frac{\partial r_j}{\partial x_1} &\dots &   \frac{\partial r_j}{\partial x_n} \\

  \end{pmatrix}
\rightarrow
\begin{pmatrix}
\begin{pmatrix}
    \frac{\partial r_1}{\partial x_1}  & \dots & \frac{\partial r_1}{\partial x_n}  \\
  \vdots & \ddots &  \vdots \\
   
   \frac{\partial r_j}{\partial x_1} &\dots &   \frac{\partial r_j}{\partial x_n} \\
\end{pmatrix}
& \bigzero\\

\begin{pmatrix}
    \frac{\partial s_1}{\partial x_1}  & \dots & \frac{\partial s_1}{\partial x_n}  \\
  \vdots & \ddots &  \vdots \\
   
   \frac{\partial s_m}{\partial x_1} &\dots &   \frac{\partial s_m}{\partial x_n} \\
\end{pmatrix}
& \begin{pmatrix}
    \frac{\partial s_1}{\partial z_1}  & \dots & \frac{\partial s_1}{\partial z_m}  \\
  \vdots & \ddots &  \vdots \\
   
   \frac{\partial s_m}{\partial z_1} &\dots &   \frac{\partial s_m}{\partial z_m} \\
\end{pmatrix}\\
  \end{pmatrix}\xrightarrow{\mathfrak{a}\gamma}
  \begin{pmatrix}
      A & \bigzero\\
      \bigstar & B\\
  \end{pmatrix}$$
  
  When choosing $n+m-1$ rows, we will only obtain a nonzero determinant by choosing all of the bottom $m$ rows, i.e.\ all the rows of the $m\times m$ matrix $B$. Otherwise, we obtain a minor of the form 
$  \begin{pmatrix}
      X & \bigzero\\
      \bigstar & Y\\
  \end{pmatrix}$, where $X$ is an $(n-1)\times(n-1)$ matrix and $Y$ is an $m\times m$ matrix with a row of zeroes, so the determinant of this minor is $|X|\cdot|Y|=|X|\cdot 0=0$.
So any minor of the right size with nonzero determinant is of the form
$  \begin{pmatrix}
      A^\prime & \bigzero\\
      \bigstar & B\\
  \end{pmatrix}$, where $A^\prime$ is a square $(n-1)$-minor of $A$. The determinant of this minor is $|A^\prime|\cdot|B|$. Since the Alexander ideal of $K_0$ is generated by exactly the determinants of these $A^\prime$, we have that $\Delta(K_1)=(|B|)\Delta(K_0)$.

\end{proof}

\end{subsection}
\end{section}



\begin{section}{The ideal class monoid of $\mathbb{Z}[t^{\pm1}]$} \label{sec:icm}
In this section we define the ideal class monoid of a ring and prove some fundamental properties in the case of $\mathbb{Z}[t^{\pm1}]$. Let $R$ be an integral domain. The set of nonzero ideals of $R$, denoted $\mathbb{I}(R)$, forms a commutative monoid under ideal multiplication. Say $I\sim J$ if there exist nonzero $x,y\in R$ such that $(x)I=(y)J$. The quotient monoid $\mathbb{I}(R)/\sim$ is called the \textbf{ideal class monoid} of $R$, denoted $\mathcal{I}(R)$. The identity element of this monoid is precisely the set of principal ideals of $R$. Hence an ideal class $[I]$ is nontrivial if and only if any representative $I$ is not principal.

A useful characterization of this equivalence relation is that $I\sim J$ if and only if $I\cong J$ as an $R$-module. Therefore the minimal number of generators of an ideal $I$ is an invariant of its ideal class. The way we will produce an infinite rank submonoid of $\mathscr{K}_0$ is by showing that any set of maximal ideals independently generates a free commutative submonoid of $\mathcal{I}(\mathbb{Z}[t^{\pm1}])$, and then find an infinite family of 2-knots with distinct, maximal Alexander ideals. The main goal of this section is to prove the statement about the ideals in detail.

The group of units of $\mathcal{I}(R)$ is the set of ideal classes $[I]$ such that there exists a class $[J]$ so that $[IJ]=[(1)]$; i.e.\ so that $IJ$ is principal. This is called the $\textbf{Picard group}$ of $R$, denoted $Pic(R)$, and whenever $R$ is a Noetherian UFD it is trivial. This will enable us to prove that any surface knot $K$ with nonprincipal Alexander ideal is not invertible in $\mathscr{C}_0$ (where by $K$ is {\it invertible} we mean that for any surface knot $J$, $K\# J$ is not 0-concordant to the unknotted surface of the same genus). We refer the reader to Section 20 of \cite{matsumura} for details. In brief, the {\it divisor class group} $C(R)$ of a Krull ring $R$ is trivial if and only if $R$ is a UFD. As $R=\mathbb{Z}[t^{\pm1}]$ is Noetherian and integrally closed, it is Krull, and of course it is a UFD as well. Furthermore, when $R$ is a Krull domain, $Pic(R)$ is naturally a subgroup of $C(R)$, hence is trivial for any Noetherian UFD.

\begin{corollary} \label{cor:pic}
No nontrivial ideal class of $\mathcal{I}(\mathbb{Z}[t^{\pm1}])$ is invertible, i.e.\ for any nonprincipal ideal $I$ and any nonzero ideal $J$, $IJ$ is not principal.
\end{corollary}

Now we turn our attention to the minimal number of generators of an ideal $I\subseteq\mathbb{Z}[t^{\pm1}]$. One elementary observation is that if $I$ is a proper ideal and $|R/I|$ is finite, $I$ cannot be principal. Indeed, if $I=(n)$, $n\geq 2$, then $R/I\cong \mathbb{Z}_n[t^{\pm1}]$ is infinite because it has polyomials of arbitrary degree. On the other hand, if $I=(f(t))$, where $deg(f)\geq 1$, then $R/I$ is infinite because it has $\mathbb{Z}$ as a subring. This gives a quick test to check if an ideal $I$ is nonprincipal (see Corollary \ref{cor:det}), but to distinguish nontrivial ideal classes from each other we will need more sophisticated tools.

Note that $R=\mathbb{Z}[t^{\pm1}]$ is a regular ring of dimension 2. This means that for any maximal ideal $m$, the localization $(R_m,m)$ is a regular local ring of dimension 2, i.e.\ the unique maximal ideal $m$ of $R_m$ is minimally generated by 2 elements. In fact, the maximal ideals of $\mathbb{Z}[t^{\pm1}]$ can be described explicitly: they are of the form $(p,f(t))$, where $p$ is a prime integer and $f(t)$ is irreducible mod $p$.

Let $m$ be a maximal ideal of $R$, and consider the localization $R_m$. If $I\subseteq R$ is an ideal, then the pushforward of $I$ is an ideal of $R_m$, denoted $I R_m$. The minimal number of generators of $I R_m$ is a lower bound for the minimal number of generators of $I$ as an ideal of $R$, since the image of a generating set of $I$ generates $I R_m$. The benefit of working in the localization is that $R_m$ is a local ring, i.e.\ it has a unique maximal ideal, $m R_m$. Now assume $(R,m)$ is a local ring. This allows some powerful techniques for computing lower bounds for the minimal number of generators of $m^n$. In this case, Nakayama's lemma implies that the minimal number of generators of $m^n$ is equal to the minimal number of generators of $m^n/m^{n+1}$. Since $m$ annihilates this $R$-module, it is a vector space over the field $R/m$, so its minimal number of generators is equal to its dimension. In general, if $M$ is a finitely generated $R$-module, the {\bf Hilbert function} $H_M(n)$ of $M$ is:
$$H_M(n):=\mathrm{dim}_{R/m} m^nM/m^{n+1}M$$
The following theorem is a combination of Theorems 1.11 and 12.1 from \cite{eisenbud}.

\begin{theorem}[Hilbert] \label{thm:num gens}
There is a polynomial $P_M(n)$, of degree $\mathrm{dim}(R)-1$, which agrees with $H_M(n)$ for sufficiently large $n$.
\end{theorem}

We are interested in the case $H_R(n)=\mathrm{dim}_{R/m} m^n/m^{n+1}$, where $R$ is the localization of $\mathbb{Z}[t^{\pm1}]$ at a maximal ideal $m$. The dimension of such an $R$ is 2 ($=\mathrm{dim}(\mathbb{Z}[t^{\pm1}])$), so $P_R(n)$ is a linear polynomial, which after some $N>0$ agrees with $H_R(n)$. Thus the minimal number of generators of $m^n/m^{n+1}$, and therefore of $m^n$, eventually agrees with a linear polynomial. A priori the minimal number of generators of $m^n$ as an ideal of $\mathbb{Z}[t^{\pm1}]$ may not agree with these values, but is certainly bounded below by them.

\begin{corollary} \label{cor:numgens}
If $m\subseteq\mathbb{Z}[t^{\pm1}]$ is a maximal ideal, then the minimal number of generators of $m^n$ grows arbitrarily large as $n$ approaches infinity.
\end{corollary}

\begin{corollary} \label{cor:maxideals}
Let $R=\mathbb{Z}[t^{\pm1}]$. The maximal ideals of $R$ form a basis for a free commutative submonoid of $\mathcal{I}(R)$, isomorphic to $\mathbb{N}^\infty$.
\end{corollary}

\begin{proof}
Let $m_1$, \dots, $m_n$ be any finite set of maximal ideals in $R$. The claim to be proved is that for any two vectors $\boldsymbol{i}=(i_1,\dots,i_n)$, $\boldsymbol{j}=(j_1,\dots,j_n)$, 
\newline $m_1^{i_1}m_2^{i_2}\cdots m_n^{i_n}\sim m_1^{j_1}m_2^{j_2}\cdots m_n^{j_n}$ implies $\boldsymbol{i}=\boldsymbol{j}$. Suppose on the contrary that the ideals are related but $\boldsymbol{i}\neq\boldsymbol{j}$, so there exist $f,g\in\mathbb{Z}[t^{\pm1}]$ such that  
\begin{equation*}
    (f)m_1^{i_1}m_2^{i_2}\cdots m_n^{i_n} =(g) m_1^{j_1}m_2^{j_2}\cdots m_n^{j_n} \tag{$*$}
\end{equation*}
and $k$ so that $i_k\neq j_k$. Now localize at $m_k$. The equation $(*)$ pushes forward to the equation $(f)m_k^{i_k}=(g)m_k^{j_k}$ $(\dagger)$ in $R_{m_k}$, since all other $m_\alpha$ contain an element in the complement of $m_k$. $(R_{m_k},m_k)$ is a local ring of dimension 2, so by Theorem \ref{cor:numgens} there exists $N>0$ such that for all distinct $\alpha$, $\beta\geq N$, $m_k^\alpha$ and $m_k^\beta$ have a different minimal number of generators. Multiply both sides of $(\dagger)$ by $m^N$ to obtain $(f)m_k^{N+i_k}=(g)m_k^{N+j_k}$. Since the left hand side and right hand side of this equation have the same minimal number of generators as $m_k^{N+i_k}$ and $m_k^{N+j_k}$, respectively, this is a contradiction. 

This proves that the maximal ideals generate a free commutative submonoid of $\mathcal{I}(R)$. Since there are infinitely many maximal ideals, this submonoid is isomorphic to $\mathbb{N}^\infty$.
\end{proof}

\begin{remark}
Restricting to maximal ideals may seem rather restrictive; however in terms of $\mathcal{I}(\mathbb{Z}[t^{\pm1}])$ it is the same as restricting to ideals which admit prime factorizations. This is because every height 1 prime ideal in $\mathbb{Z}[t^{\pm1}]$ is principal, so the only nonprincipal prime ideals are height $2=\mathrm{dim}(\mathbb{Z}[t^{\pm1}])$, hence are maximal. So, as long as an ideal admits a prime factorization, we can pin down its ideal class uniquely by looking at the multiplicities of the maximal ideals in that factorization.
\end{remark}

\begin{remark}
There is another, in some sense easier, way to prove Corollary \ref{cor:maxideals}. One can show that, in a Noetherian domain $R$: if an ideal $I$ admits a prime factorization, then that factorization is unique. Then, by a similar localization argument one quickly shows that distinct products of maximal ideals lie in different ideal classes. We included the previous argument because the minimal number of generators of an ideal, though hard to compute, gives more of a quantitative sense of how ideal classes can differ than simply resorting to uniqueness of prime factorizations. Also, our main corollary applies to surface knots with nonprincipal ideals, so by establishing that there are 2-knots whose ideals require arbitrarily many generators we are putting this requirement in some perspective. Classical knots (and ribbon 2-knots) have principal Alexander ideals for the special reason that they have deficiency 1 Wirtinger presentations, while Levine showed in \cite{levine} that a 2-knot group can have any deficiency less than 1 (see also \cite{kanenobusat}). Certainly a 2-knot taken `at random' should not be expected to have a deficiency 1 knot group nor a principal Alexander ideal.

\end{remark}

\end{section}
\begin{section}{0-Concordance and Alexander ideals}
In this section we prove the main theorem and applications. Recall that $\mathscr{C}_0$ denotes the monoid of oriented surface knots in $S^4$ modulo 0-concordance. The 0-concordance monoid of 2-knots, $\mathscr{K}_0$, is a submonoid of $\mathscr{C}_0$. A surface knot $K$ is {\it 0-slice} if it is 0-concordant to the unknotted surface of the same genus, and {\it invertible} if there exists a surface knot $J$ so that $K\# J$ is 0-slice. Note that this is looser than the usual meaning of invertibility; indeed only a genus 0 surface has a chance at having a true inverse. As a warmup to the main theorem, we carry out an example from first principles.

\begin{example}
Let $K$ be the 2-twist spun trefoil. Then $\Delta(K)=(3,t+1)$ is maximal, as shown in Section \ref{sec:fox}, hence minimally generated by 2 elements. Suppose that $K$ is 0-concordant to the unknot $\mathcal{U}$. Then there exists a 2-knot $J$ and ribbon concordances $K\rightarrow J\leftarrow \mathcal{U}$. Since $J$ is ribbon concordant to a ribbon knot, $J$ is ribbon. On the other hand, by the key lemma $K\rightarrow J$ implies $\Delta(J)=(f)\Delta(K)=(f)(3,t+1)$ for some nonzero $f\in R$. Notice $(f)(3,t+1)\cong(3,t+1)$ as an $R$-module, therefore has the same minimal number of generators. Thus $\Delta(J)$ is not principal, but $J$ was supposed to be ribbon.
\end{example}

\homomorphism*


\begin{proof}
 Suppose $K_0$ is 0-concordant to $K_1$. Then by Proposition \ref{prop:factor} there exists a surface knot $J$ with ribbon concordances $K_0\rightarrow J\leftarrow K_1$. So, by the key lemma there exist $f,g\in \mathbb{Z}[t^{\pm1}]$ such that $(f)\Delta(K_0)=\Delta(J)=(g)\Delta(K_1)$, thus $\Delta(K_0)$ and $\Delta(K_1)$ are equivalent in $\mathcal{I}(\mathbb{Z}[t^{\pm1}])$. As shown in Proposition \ref{prop:sum}, $\Delta(K\# J)=\Delta(K)\Delta(J)$, so the map $[K]\rightarrow [\Delta(K)]$ is a homomorphism.
\end{proof}

Since an ideal class is nontrivial if and only if it consists of nonprincipal ideals, this gives an easily computable obstruction to being 0-slice. In fact, since the group of units of $\mathcal{I}(\mathbb{Z}[t^{\pm1}])$ is trivial (see Corollary \ref{cor:pic}), any surface knot with nonprincipal Alexander ideal is not invertible in $\mathscr{C}_0$.

\noninvertible*

Twist spun knots provide many examples of 2-knots with nonprincipal ideals. Together with the previous corollary, this proves that the 0-concordance monoid of 2-knots, $\mathscr{K}_0$, is not a group. 

\inftytwisty*

\begin{proof}
First we compute the Alexander ideal of $\tau^n K$. Let $\langle x_0,\dots,x_m|r_1,\dots,r_m\rangle$ be a Wirtinger presentation for $\pi K$. As discussed in Section \ref{sec:twistspin}, 
\newline $\langle x_0,\dots,x_m|r_1,\dots,r_m,[x_0^n,x_1],\dots,[x_0^n,x_m]\rangle$ is a Wirtinger presentation for $\pi (\tau^n K)$. The Alexander matrix calculated from this presentation is equivalent to the following matrix, which we get by replacing the first column with zeroes as before.

$$\begin{pmatrix}
    0 & a_{11} & \dots & a_{1m}\\
    \vdots & \vdots & \ddots & \vdots\\
    0 & a_{m1} & \dots & a_{mm}\\
    0  \\
    \vdots & & (t^n-1)I_m \\
    0 & \\
    
\end{pmatrix}$$

By deleting the first column altogether and remembering to take determinants of minors of size $m$, we arrive at the following convenient form 
$\begin{pmatrix}
A \\
(t^n-1)I_m
\end{pmatrix}$. Note that $|A|=\Delta_K(t)$ is the generator of $\varepsilon_1(K)=\Delta(K)$. Then the Alexander ideal is $$\displaystyle{\sum_{j=1}^{m+1}((t^n-1)^{j-1})\varepsilon_j(K)}=(\Delta_K(t),(t^n-1)\varepsilon_2(K),\dots,(t^n-1)^{m-1}\varepsilon_{m}(K),(t^n-1)^m)$$ (recall that $\varepsilon_j(K)$ need not be principal and that $\varepsilon_{m+1}(K)$ is $(1)$ by definition). What we need here is that $\Delta_K(t)$ and $(t^n-1)^m$ are in $\Delta(\tau^n K)$.

We will actually prove that such a $K$ has infinitely many {\it even} twist spins with nonprincipal ideal (cf Proposition \ref{prop:twistdet}). Suppose $n$ is even and $\Delta(\tau^n K)=(f_n(t))$ is principal. Evaluating the above equation at $t=1$ we obtain $(f_n(1))=(\Delta_K(1))=(1)$, and at $t=-1$, $(f_n(-1))=(\Delta_K(-1))\neq(1)$ by assumption. Therefore $f_n$ has degree at least one. 

Since $\Delta_K$ and $(t^n-1)^m$ are in $(f_n)$, there exist $g_n,h_n$ such that $\Delta_K=g_nf_n$ and $(t^n-1)^m=h_nf_n$. The second equation implies that all the roots of $f_n$ are $n^\text{th}$ roots of unity. The first equation implies that they are also roots of $\Delta_K$. This proves the theorem. 

Note that $n=2$ always works, since $\Delta_K(\pm1)$ is odd. If we list the primitive $m_i^\text{th}$ roots of unity which are roots of $\Delta_K$, then as long as $0\neq n\in2\mathbb{Z}\setminus\{km_i:k\in\mathbb{Z}\}$, we are guaranteed that $\Delta(\tau^n K)$ is nonprincipal, and thus $\tau^n K$ is not invertible in $\mathscr{C}_0$. In particular, if $\Delta_K$ has no roots of unity as roots, then all of its nonzero even twist spins are not invertible.

\end{proof}

\determinant*

\begin{proof}

Notice that $(f(t),t-a)=(f(a),t-a)$. This is because $f(t)-f(a)$ is divisible by $t-a$. When $K$ is a 2-bridge knot, $\Delta(\tau^2 K)=(\Delta_K(t),t+1)=(\Delta_K(-1),t+1)$, so the ideal of the 2-twist spin of $K$ is generated by $t+1$ and the determinant of $K$. When $|\Delta_K(-1)|=n>1$, the quotient $\mathbb{Z}[t^{\pm1}]/(n,t+1)\cong\mathbb{Z}/n\mathbb{Z}$ is finite and nonzero, hence $(n,t+1)$ is not principal.
The proof will be finished once we establish the following claim.

Claim: Let $n,m\geq 0$ be odd integers. If $(n,t+1)\sim(m,t+1)$, then $n=m$.

Suppose the ideals are related, then there exist $f,g$ so that $(f)(n,t+1)=(g)(m,t+1)$. Localize by inverting the multiplicative set $\{(t+1)^k|k\geq0\}$: in the localization, this equation becomes $(\overline{f})=(\overline{g})$. Since $t+1$ is irreducible, there exist $j,k$ so that in $\mathbb{Z}[t^{\pm1}]$, $((t+1)^jf)=((t+1)^kg)$. Multiplying the original equation by $(t+1)^j$, we see that 
$$((t+1)^jf)(n,t+1)=((t+1)^kg)(n,t+1)=((t+1)^jg)(m,t+1)$$
Evaluating both sides of this equation at $t=1$, we obtain:
$$( 2^k\cdot g(1)) (1)=(2^j\cdot g(1))(1)$$
Therefore $j=k$. Then $((t+1)^jg)(n,t+1)=((t+1)^jg)(m,t+1)$ implies $(n,t+1)=(m,t+1)$, which implies $n=m$ (by looking at quotients, or by evaluating at $t=-1$).

\end{proof}

\begin{remark} \label{rem:comparison}
The Stevedore knot $6_1$ has determinant 9, so $\Delta(\tau^2 6_1)=(2t^2-5t+2,t+1)=(9,t+1)$ is not principal. Since the Stevedore knot is slice, its double branched cover (which is the fiber of its 2-twist spin) is spin rational homology cobordant to $S^3$. This can be seen by capping off a concordance $(S^3\times I,C)$ between the Stevedore and the unknot with $(B^4, \text{Seifert surface})$ pairs on both sides to get a closed surface in $S^4$, then taking the double branched cover of $S^4$ over this surface. This is a spin 4-manifold, and by restricting to the relevant pieces we get a spin rational homology cobordism from the double branched cover of the Stevedore knot to $S^3$. Also, the double branched cover of a knot in $S^3$ has a unique spin structure, so this agrees with the one induced by $S^4$ on the fiber of the 2-twist spin of the Stevedore. Thus this is an example where the techniques of \cite{sunukjian0conc}, \cite{daimiller} cannot obstruct 0-concordance, but Alexander ideals can. There are infinitely many slice 2-bridge knots with any given nonunit, square determinant, so all of their double branched covers share this property. Conversely, the 5-twist spun trefoil has $\Delta(\tau^5 3_1)=(1)$, but a Seifert solid (the Poincar\'{e} homology sphere) with nonzero $d$-invariant \cite{sunukjian0conc}. Dai-Miller also produce many examples where their invariant distinguishes 0-concordance but the Alexander ideal is trivial. This shows that the homomorphism $\Delta$ is not injective. We will determine its image in Section \ref{sec:ideals}.
\end{remark}

We turn now to identify an infinite rank submonoid of $\mathscr{K}_0$. 

\bigmonoid*

\begin{proof}
In Corollary \ref{cor:maxideals} we showed that any set of maximal ideals is linearly independent in $\mathcal{I}(\mathbb{Z}[t^{\pm1}])$. Therefore any set of 2-knots with distinct, maximal Alexander ideals is linearly independent in $\mathscr{C}_0$, by Theorem \ref{thm:hom}.

Let $K_p$ be any 2-bridge knot with prime determinant $p=\Delta_{K_p}(-1)$. Then the 2-twist spin of $K_p$ has maximal Alexander ideal: $\Delta(\tau^2 K_p)=(\Delta_{K_p}(-1),t+1)=(p,t+1)$, as in Corollary \ref{cor:2twist}. For instance, $K_p$ could be the 2-twist spin of the $(2,p)$-torus knot. For any set of odd prime numbers $\{p_i\}$, the corresponding 2-knots $\tau^2 K_{p_i}$ form the basis for a linearly independent submonoid of $\mathscr{K}_0$.

\end{proof}

\begin{example}
Another interesting family is the $p$-twist spins of $(2,p)$-torus knots, with $p$ an odd prime. The Alexander ideal of $\tau^p T(2,p)$ is $I_p=(\Phi_{2p}(t),\Phi_p(t))$. Note this is equal to $(2,\Phi_p(t))$, since $2=(1+t)\Phi_{2p}(t)+(1-t)\Phi_p(t)$, and $\Phi_{2p}(t)=\Phi_p(t)-2(t^{p-2}+\dots+t^3+t)$. The quotient $\mathbb{Z}[t^{\pm1}]/I_p$ has order $2^{p-1}$, so none of these 2-knots are 0-slice. 
When 2 is a primitive root mod $p$, $\Phi_{p}(t)$ is irreducible mod 2, so $I_p$ is maximal. If the Artin conjecture is true, then 2 is a primitive root for infinitely many primes $p$, so this would give an infinite basis for another linearly independent family. It would also show that one can obtain finite fields of order $2^k$ for arbitrarily large $k$ as $\mathbb{Z}[t^{\pm1}]/\Delta(K)$ with $K$ a 2-knot.

\end{example}

\end{section}


\begin{section}{Alexander ideals of Knotted surfaces} \label{sec:ideals}
\begin{subsection}{Realizability}\

In 1960, Kinoshita proved that any polynomial $f(t)\in\mathbb{Z}[t^{\pm1}]$ with $f(1)=\pm 1$ is the Alexander polynomial of a ribbon 2-knot \cite{kinoshita}. Here we strengthen this theorem by achieving the same result with 2 generator, 1 relator Wirtinger presentations. This allows us to generalize Kinoshita's theorem to a complete characterization of which ideals occur as the Alexander ideals of surface knots. If $I\subseteq\mathbb{Z}[t^{\pm1}]$ is an ideal and $a\in\mathbb{Z}$, let $I|_{t=a}$ denote the nonnegative generator of $\{f(a):f(t)\in I\}\subseteq\mathbb{Z}$.

\ideal*

\begin{theorem} \label{lem:poly}
If $f(t)\in\mathbb{Z}[t^{\pm1}]$ satisfies $f(1)=\pm 1$, then there is a ribbon 2-knot $K$ of meridional rank 2 with $\Delta(K)=(f(t))$.
\end{theorem}
\begin{proof}
Our basic strategy is the same as Kinoshita's, but we achieve $f(t)$ as the Fox derivative of a single Wirtinger relator as opposed to the determinant of a large matrix. We will construct a Wirtinger presentation $\langle x,y|r\rangle$, where $r=xwy^{-1}w^{-1}$ for some word $w\in\langle x,y\rangle$. Any such presentation presents the knot group $\pi K$ of a ribbon 2-knot $K$ (see \cite{kawauchi}). The Jacobian is 
$\displaystyle{ \begin{pmatrix}
    \frac{\partial r}{\partial x}  & \frac{\partial r}{\partial y}  \\
  \end{pmatrix}}$.

Let $r_x$ denote $\frac{\partial r}{\partial x}$. Since this is a Wirtinger presentation, $x$ and $y$ abelianize to $t$ as before. Likewise, after abelianization the sum across the row is zero, so $\mathfrak{a}\gamma(r_y)=-\mathfrak{a}\gamma(r_x)$. The Alexander ideal of $K$ is then generated by the abelianization of $r_x=1+xw_x-xwy^{-1}w^{-1}w_x$, i.e.\ $\Delta(K)=(1+\mathfrak{a}\gamma (w_x)(t-1))$.

Since the Alexander polynomial is only defined up to a unit, we may assume $f(1)=1$, so that $f(t)=1+g(t)(t-1)$ for some polynomial $g(t)$. We will show that $w$ can be chosen so that $\mathfrak{a}\gamma(w_x)=g$, i.e.\ so that the abelianized matrix is $\begin{pmatrix}
f(t) & -f(t)\\ 
\end{pmatrix}$ and $\Delta(K)=(f(t))$. 

For clarity, we first point out that if $w=y^{n_1} x^{n_2}\dots y^{n_{2k -1}}x^{n_{2k}}$, $n_i\in\mathbb{Z}$ then 
$$f(t)=1+t^{n_1}(t^{n_2}-1)+t^{n_1+n_2+n_3}(t^{n_4}-1)+\dots+ t^{n_1+\dots + n_{2k-1}}(t^{n_{2k}}-1),$$

as can be checked directly (note that $\frac{dx^n}{dx}(x-1)=x^n-1$ for all $n\in\mathbb{Z}$). We will only need the case $n_{2i}=\pm 1$, and add the desired terms $ t^{n_i}(t-1)$ and $ t^{n_i}(t^{-1}-1)=- t^{{n_i}-1}(t-1)$ as many times as needed. In particular, if 
$$f(t)=1+t^{m_1}(t^{m_2}-1)+t^{m_3}(t^{m_4}-1)+\dots+ t^{m_{2k-1}}(t^{m_{2k}}-1),$$
then by letting 
$\begin{cases}
n_1=m_1&\\ 
n_{2i}=m_{2i} & i\geq 1\\
n_{2i+1}=m_{2i+1}-(m_{2i-1}+m_{2i}) & i\geq 1\\
\end{cases}$
\vspace{.2cm}

\noindent we arrive at the desired form for $w=y^{n_1} x^{n_2}\dots y^{n_{2k -1}}x^{n_{2k}}$.

\end{proof}
\begin{remark}
In \cite{kawauchi}, Theorem 13.5.3, it is shown that any Alexander polynomial can be obtained from a 2 generator, 1 relator presentation. However, these are not Wirtinger presentations, so it is not obvious how to add more relations to achieve any desired set of generators, as we do next.
\end{remark}

\begin{proof}[Proof of Theorem \ref{thm:ideal}]
Let $I$ be an ideal such that $I|_{t=1}=1$. Since $\mathbb{Z}[t^{\pm1}]$ is Noetherian, $I$ admits a finite generating set $g_1(t),\dots,g_m(t)$, so by assumption\newline $(g_1(1),\dots, g_m(1))=(1)$. This implies there is a linear combination $f_0(t)=\sum a_i g_i(t)$, $a_i\in\mathbb{Z}$, such that $f_0(1)=1$. Let $f_i(t)=g_i(t)-(g_i(1)-1)f_0(t)$, $1\leq i\leq m$. Then $I=(g_1(t),\dots,g_m(t))=(f_0(t), g_1(t),\dots,g_m(t))=(f_0(t),f_1(t),\dots,f_m(t))$ has a generating set such that each generator evaluates to 1 at 1.

Then, as in Theorem \ref{lem:poly}, we can build a relation $r_i$ for each $f_i$ to obtain a Wirtinger presentation $\langle x,y|r_0,\dots,r_m\rangle$ for a knot group $G$, from which we can construct a genus $m$ ribbon surface knot $K$ with $\pi K \cong G$. To do this, first construct a 2-knot $K_0$ with presentation $\langle x,y| r_0\rangle$, then attach a 1-handle to $K_0$ for each additional relation. The resulting Alexander matrix is $(m+1)\times 2$, and the Alexander ideal is then $\Delta(K)=(\mathfrak{a}\gamma(r_{0x}), \dots, \mathfrak{a}\gamma(r_{mx}))=(f_0(t),\dots,f_m(t))=I$.

\vspace{.2cm}
It remains to show that if $K$ is a surface knot, then $\Delta(K)|_{t=1}=1$. This will follow from the observation that we can compute $\Delta(K)|_{t=1}$ by evaluating the entries of the Alexander matrix at 1 before taking determinants of minors, and by using an especially nice Wirtinger presentation for $\pi K$.

Let $P$ be a Wirtinger presentation for $\pi K$. Since $K$ is connected, all of the generators $x_0,\dots, x_n $ of $P$ are conjugate. Therefore we can rewrite the relations to obtain a presentation of the form $\langle x_0,\dots, x_n |r_1, \dots r_n, s_1,\dots,s_m \rangle$, such that $r_i=x_iw_ix_0^{-1}w_i^{-1}$, $1\leq i\leq n$, and $s_i=x_0w_ix_0^{-1}w_i^{-1}$, $1\leq i\leq m$. If $r$ is any Wirtinger relation, then $\mathfrak{a}\gamma(r_{x_i})|_{t=1}$ is equal to the exponent sum of $x_i$ in $r$ (this was already shown in one case in the proof of Theorem \ref{lem:poly}. The other cases are similar, but note that $\mathfrak{a}\gamma(w)=t^N$ for some $N$, since $w$ is an arbitrary word in $x$ and $y$. So $\mathfrak{a}\gamma(w)|_{t=1}=1$). Therefore, when we form the Alexander matrix for $K$ and evaluate the entries at 1, we obtain the matrix:
$$
  \begin{pmatrix}
    r_{1 x_0}  & \dots & r_{1 x_n}  \\
    r_{2 x_0}  & \dots &  r_{2 x_n}  \\
  \vdots & \ddots &  \vdots \\
    r_{n x_0}&\dots& r_{n x_n}\\
       s_{1 x_0}  & \dots &  s_{1 x_n}  \\
  \vdots & \ddots &  \vdots \\
    s_{m x_0} &\dots &    s_{m x_n} \\

  \end{pmatrix}
  \rightarrow\begin{pmatrix}
-1  & 1 &  0 & \dots & 0 & 0\\
-1 & 0 & 1 &  \dots & 0 & 0\\
-1 & 0 & 0 & \ddots & \vdots & \vdots\\
\vdots & \vdots & \vdots &   & 1 & 0\\
-1 & 0 & 0 &  \dots & 0 & 1\\
0 & \dots & 0 & \dots& 0 & 0\\
\vdots & \ddots & \vdots & \ddots & \vdots & \vdots\\
0 & \dots & 0 & \dots & 0 & 0\\

\end{pmatrix}$$
Since the identity matrix is an $(n\times n)$-minor, $\Delta(K)|_{t=1}=1$.

\end{proof}

\begin{corollary} \label{cor:highergenus}

Let $I$ be an ideal of $\mathbb{Z}[t^{\pm1}]$ such that $I|_{t=1}=1$. If $I$ is minimally generated by $g$ elements, then there is a ribbon surface knot $K$ of genus $g$ and meridional rank 2 with $\Delta(K)=I$.
\end{corollary}

In particular, any maximal ideal $m$ is minimally generated by two elements $f(t),g(t)$. The gcd of $f(1)$ and $g(1)$ is 1 if and only if $m$ is the ideal of a surface knot, and the above construction yields a surface of genus 2 realizing this ideal. If $m$ can be written as $(f(t),g(t))$ such that $f(1)=1$, then $m$ is the ideal of a knotted torus.

Let $\mathbb{I}_K=\{I\subseteq\mathbb{Z}[t^{\pm1}]:I|_{t=1}=1\}$, i.e.\ $\mathbb{I}_K$ is the set of surface knot ideals. Let $\mathcal{I}_K=\{[I]:I\in\mathbb{I}_K\}$ be the submonoid of $\mathcal{I}(\mathbb{Z}[t^{\pm1}])$ of classes with a representative in $\mathbb{I}_K$. This is manifestly the image of the homomorphism $\Delta:\mathscr{C}_0\rightarrow \mathcal{I}(\mathbb{Z}[t^{\pm1}])$.

\image*

There is another monoid one might consider. Its construction is the same as the ideal class monoid, but we restrict the equivalence relation $\sim$ to $\mathbb{I}_K$: for ideals $I,J\in\mathbb{I}_K$, $I\sim_K J$ if there exist $f(t),g(t)\in\mathbb{Z}[t^{\pm1}]$ such that $f(1)=1=g(1)$ (which is the same as requiring $(f),(g)\in\mathbb{I}_K$) and $(f)I=(g)J$. Then $\mathbb{I}_K/\sim_K$ is the monoid of interest. In fact this is isomorphic to the monoid $\mathcal{I}_K$ above. In the ideal class monoid, $I\sim J$ if there exist $f(t),g(t)\in\mathbb{Z}[t^{\pm1}]$ so that $(f)I=(g)J$, and this is equivalent to the existence of a $\mathbb{Z}[t^{\pm1}]$-module isomorphism $\phi:I\rightarrow J$. The less obvious direction of the equivalence is that if $\phi$ is such an isomorphism, then for any $f\in I$, $(\phi(f))I=(f)J$. When $I\in\mathbb{I}_K$, $I$ contains some element $f$ such that $f(1)=1$, so $I\cong J$ implies $I\sim_K J$. Thus $I\sim J$ if and only if $I\sim_K J$, and $\mathcal{I}_K$ is canonically isomorphic to $\mathbb{I}_K/\sim_K$.

A harder question is determining the image when restricted to 2-knots. In Corollary \ref{cor:nonreversible} we will point out out some ideals in $\mathbb{I}_K$ which, due to a theorem of Guti\'errez, are not the Alexander ideals of any 2-knots; it seems likely that these ideal classes are also missed in the image of $\Delta$ restricted to the 2-knot monoid $\mathscr{K}_0$.

When $I\in\mathbb{I}_K$ and $f(1)=1$, we showed how to construct a surface knot $K$ with $\Delta(K)=I$ and a ribbon 2-knot $J$ with $\Delta(J)=(f)$. Connect summing with a ribbon 2-knot is always a ribbon concordance, so we have a ribbon concordance $K\rightarrow K\# J$ realizing these ideals. 

\begin{question} \label{q:effectiveness}
If $I_0,I_1\in\mathbb{I}_K$ and $I_0\sim I_1$, do there exist 0-concordant surface knots $K_0$, $K_1$ such that $\Delta(K_0)=I_0$, $\Delta(K_1)=I_1$?
\end{question}
If there exists an ideal $J$ such that $I_0=(f)J$, $I_1=(g)J$ for some $f$, $g$, then this is clearly true: by Theorem \ref{thm:ideal} there is a surface knot $K$ such that $\Delta(K)=J$, and by Theorem \ref{lem:poly} there are ribbon 2-knots $K_f$, $K_g$ with $\Delta(K_f)=(f)$, $\Delta(K_g)=(g)$. Thus there are ribbon concordances

$$\begin{tikzcd}
& K\# K_f \# K_g\\
K\# K_f \arrow[ru] & & K\# K_g\arrow[lu]\\
& K\arrow[lu] \arrow[ru]
\end{tikzcd}$$
so $K_0=K\# K_f$ is 0-concordant to $K_1=K\# K_g$, and $\Delta(K_j)=I_j$.

When there exists no such ideal $J$, the situation is unclear; however we do not know of any ideals $I_0\sim I_1$ for which this is the case.

\end{subsection}

\begin{subsection}{Inversion of $t$}\label{tinversion}\

Reversing the orientation of a surface knot $K$ amounts to changing $t$ to $t^{-1}$ in $H_1(S^4\setminus K)$. This change is not detected by the Alexander polynomial of a classical knot, since these polynomials are all symmetric. For surface knots this is not the case, and in fact the ideal class of a surface knot can be distinct from that of its reverse.
\nonreversible*

\begin{proof}
Let $I=(f(t),p)$ be an ideal of $\mathbb{Z}[t^{\pm1}]$ satisfying:

$\begin{array}{ll}
    1) & \text{$p$ is prime and $f(t)$ is irreducible mod }p. \\
    2) & f(t)\text{ is not symmetric mod $p$.}\\
    3) & f(1)=1.
\end{array}$

Then $I=\Delta(K)$ for a ribbon torus knot $K$ which is not 0-concordant to $-K$. Since $f(1)=1$, $I=(f(t),p-(p-1)f(t))$ is of the right form to apply Corollary \ref{cor:highergenus} and build a ribbon torus $K$ with $\Delta(K)=I$ (if $f(1)\neq 1$ but $f(1)$ and $p$ are coprime, then a genus 2 surface not 0-concordant to its reverse can be constructed). Condition 1 guarantees that $I$ is maximal. Reversing the orientation of $K$ has the effect of changing $t$ to $t^{-1}$, so condition 2 guarantees that $\Delta(K)=(f(t),p)\neq(f(t^{-1}),p)=\Delta(-K)$, so these tori are not isotopic. In fact $(f(t^{-1}),p)$ is also maximal, since $t\rightarrow t^{-1}$ is an automorphism of $\mathbb{Z}[t^{\pm1}]$, so $K$ and $-K$ are not 0-concordant by Corollary \ref{cor:maxideals}.

Let $p$ be an odd prime. Then $I_p=(2t-1,p)$ satisfies 1-3, so this is one infinite family of examples. 

\end{proof}

Many more examples could be constructed along these lines. Of course, achieving the result for tori implies it for any higher genus, since adding trivial handles does not change the knot group nor the ideal. We remark that any ideal of the form $(f(t),p)$, where $p$ is prime and $(f(t))\neq(f(t^{-1}))$ as ideals of $\mathbb{Z}_p[t^{\pm1}]$, is not the ideal of a $2$-knot \cite{gutierrez}. Indeed, knot groups giving rise to these ideals are prototypical examples of 3-knot groups which are not 2-knot groups \cite{kawauchi}.

There are 2-knots whose ideals are not invariant under inverting $t$, for instance any $f(t)$ with $f(1)=1$ and $(f(t))\neq(f(t^{-1}))$ is the Alexander polynomial of such a 2-knot, but these ideals all represent the trivial ideal class. 

\begin{question}
Is there a 2-knot $K$ such that the ideal class of $\Delta(K)$ is not equal to the ideal class of $\Delta(-K)$?
\end{question} 

No twist spun knot has this property, which we prove now.

\begin{proposition} \label{twistinvert}
If $K$ is a classical knot, then $\Delta(\tau^n K)=\Delta(-\tau^n K)$.
\end{proposition}

Recall from Theorem \ref{thm:inftytwisty} that if $\langle x_0,\dots,x_m|r_1,\dots,r_m\rangle$ is a Wirtinger presentation for $\pi K$, then the Alexander ideal of $\tau^n K$ is $$\displaystyle{\sum_{j=1}^{m+1}((t^n-1)^{j-1})\varepsilon_j(K)}=(\Delta_K(t),(t^n-1)\varepsilon_2(K),\dots,(t^n-1)^{m-1}\varepsilon_{m}(K),(t^n-1)^m).$$ 

For a classical knot $K$, all of the elementary ideals $\varepsilon_j(K)$ are invariant under inversion of $t$ (Theorem 9.2.3, \cite{crowfox}). Since $(t^n-1)$ is as well, we see directly that $\Delta(\tau^n K)=\Delta(-\tau^n K)$.

\end{subsection}

\begin{subsection}{The determinant of a surface knot}\

The definition below generalizes the notion of determinant to the case of nonprincipal ideals. This will be an odd integer, which, as in the case of a single generator, follows from the fact that $\Delta(K)|_{t=1}=1$.

\begin{definition}
Let $K$ be a surface knot. The {\it determinant} of $K$ is $\Delta(K)|_{t=-1}$.
\end{definition}

\begin{proposition}
The determinant of a surface knot is odd.
\end{proposition}

\begin{proof}
Let $K$ be a surface knot. We saw in Corollary \ref{cor:highergenus} that $\Delta(K)$ has a generating set $(f_1(t),\dots,f_n(t))$ where each $f_i(1)=1$. Thus $f_i(t)=(t-1)g_i(t)+1$ for some $g_i(t)$. Therefore $f_i(-1)=-2g_i(-1)+1$ is odd for each $i$, so $\Delta(K)|_{t=-1}$, the positive generator of $(f_1(-1),\dots,f_n(-1))$, is odd as well.
\end{proof}

\begin{proposition}\label{prop:twistdet}
Let $K$ be a classical knot. The determinant of $\tau^n K$ is equal to the determinant of $K$, $|\Delta_K(-1)|$, if $n$ is even and 1 if $n$ is odd.
\end{proposition}

\begin{proof}
Recall from Theorem \ref{thm:inftytwisty} that if $\langle x_1,\dots,x_{m+1}|r_1,\dots,r_m\rangle$ is a Wirtinger presentation for $\pi K$, then $\Delta(\tau^n K)=\displaystyle{\sum_{j=1}^{m+1}((t^n-1)^{j-1})\varepsilon_j(K)}$. When $n$ is even, $(-1)^n-1=0$, so $\Delta(\tau^nK)|_{t=-1}=\varepsilon_1(K)|_{t=-1}=|\Delta_K(-1)|$, and when $n$ is odd $(-1)^n-1=-2$, so $\Delta(\tau^n K)|_{t=-1}$ has $\Delta_K(-1)$, which is odd, and $(-2)^m$ as generators, hence the determinant is 1.
\end{proof}

\begin{remark}
The previous corollary shows that this definition contains more information than evaluating the usual definition of the Alexander polynomial of a surface knot at $t=-1$. The usual definition of the $j^{\text{th}}$ Alexander polynomial of $K$ is to take a generator of the smallest principal ideal which contains $\varepsilon_j(K)$. When $K$ is a 2-bridge knot with determinant $|\Delta_K(-1)|=p>1$, the ideal of its 2-twist spin is $(\Delta_K(-1),t+1)=(p,t+1)$, and the only principal ideal which contains it is the unit ideal $(1)$. Its first Alexander polynomial is therefore $1$, as is its determinant. With our definition, however, $(\Delta_K(-1),t+1)|_{t=-1}=p\neq1$. This is the more desirable answer, since a nontrivial Fox $p$-coloring of $K$ extends to one of $\tau^2 K$ in the obvious way. 

\end{remark}

\begin{proposition}
Let $K$ be a classical knot with a nontrivial Fox $p$-coloring $\phi:\pi K\rightarrow D_p$. Then $\phi$ factors through the quotient map $\pi K\rightarrow \pi(\tau^n K)$ if and only if $n$ is even.
\end{proposition}

\begin{proof}
Let $\phi:\pi K\rightarrow D_p$ be a nontrivial $p$-coloring and $\langle x_0,\dots,x_m|r_1,\dots,r_m\rangle$ a Wirtinger presentation of $\pi K$. As discussed previously, we can form $\pi(\tau^n K)$ from $\pi K$ by adding the relations $x_0^nx_i=x_ix_0^n$, $1\leq i\leq m$, or equivalently $x_0^n=x_i^n$, $1\leq i\leq m$. The condition for $\phi$ to factor through the natural quotient map is for the images of these equations to be satisfied in $\phi(\pi K)$. We use the latter set of equations. When $n$ is even, this is automatic, since Fox colorings map meridians to reflections, which are of order 2. When $n$ is odd, the highest even power of $\phi(x_j)^n$ will vanish by the above observation, leaving $\phi(x_0)=\phi(x_i)$, $1\leq i\leq m$. Thus the only colorings which factor through the group of an odd twist spin were trivial to begin with.
\end{proof}

\end{subsection}
\end{section}


\begin{section}{0-Concordance and Peripheral subgroups}

In this section we show that under a mild condition on the knot group, the peripheral subgroup of a surface knot is also a 0-concordance invariant. All 2-knots have infinite cyclic peripheral subgroup, so this invariant is only useful for higher genus surfaces.

\begin{definition}
For a surface knot $K:\Sigma_g\hookrightarrow S^4$, the {\it peripheral subgroup $P(K)$} is the image of $i_*:\pi_1(\partial X)\rightarrow\pi_1 (X)$, where $X=\overline{S^4\setminus \nu K}$ is the exterior and $i:\partial X\hookrightarrow X$ is the inclusion.
\end{definition}

Note that $\nu K\cong\Sigma_g\times D_2$, so $\partial X\cong \Sigma_g\times S_1$. Therefore $P(K)\cong\mathbb{Z}\oplus G$, where the first factor is generated by a meridian of $K$ and the second factor is some quotient of $\pi_1(\Sigma_g)$. The unknot $\mathcal{U}_g$ of genus $g$ always has peripheral subgroup $\mathbb{Z}$. When $g=1$, it is known that $G$ can be $0,\mathbb{Z}$, $\mathbb{Z}_n$, $\mathbb{Z}\oplus\mathbb{Z}_n$, or $\mathbb{Z}\oplus\mathbb{Z}$ \cite{kankaz}. 

Recall from Proposition \ref{prop:hom} that a ribbon concordance $K_0\rightarrow K_1$ induces a homomorphism $\phi:\pi K_0\rightarrow \pi K_1$.

\begin{lemma} \label{lem:periph}
If $K_0\rightarrow K_1$ is a ribbon concordance, then the induced homomorphism $\phi: \pi K_0\rightarrow \pi K_1$ restricts to a surjection  $P(K_0)\twoheadrightarrow P(K_1)$.
\end{lemma}

The proof comes down to a diagrammatic method of writing a generating set for $P(K)$, which we now introduce. Let $D$ be a broken surface diagram for $K$. Write down the Wirtinger presentation for $\pi K$ induced by $D$. Choose a basepoint region and record the meridian corresponding to that region, say $x$. Draw a generating system of curves for $\pi_1(\Sigma_g)$ on the surface. For each such curve $\gamma$, we can write a pushoff of $\gamma$ into the exterior of $K$ in the Wirtinger generators in a manner analogous to writing the longitude of a classical knot group. First orient gamma, then traverse the curve once, starting at the basepoint. When passing through a double curve crossing while on the undersheet, write down the generator corresponding to the oversheet, raised to the sign of the crossing. The sign of the crossing is $+1$ if the normal to the oversheet agrees with the orientation of $\gamma$, and $-1$ if not. After traversing the curve, multiply by $x$ raised to the negative of the exponent sum of the word just created. If $\{\gamma_1,\dots,\gamma_{2g}\}$ is a generating system of curves, then $\langle x,\gamma_1,\dots,\gamma_{2g}\rangle\leq\pi K$ is the peripheral subgroup of $K$.

\begin{center}
\includegraphics[width=2.5in]{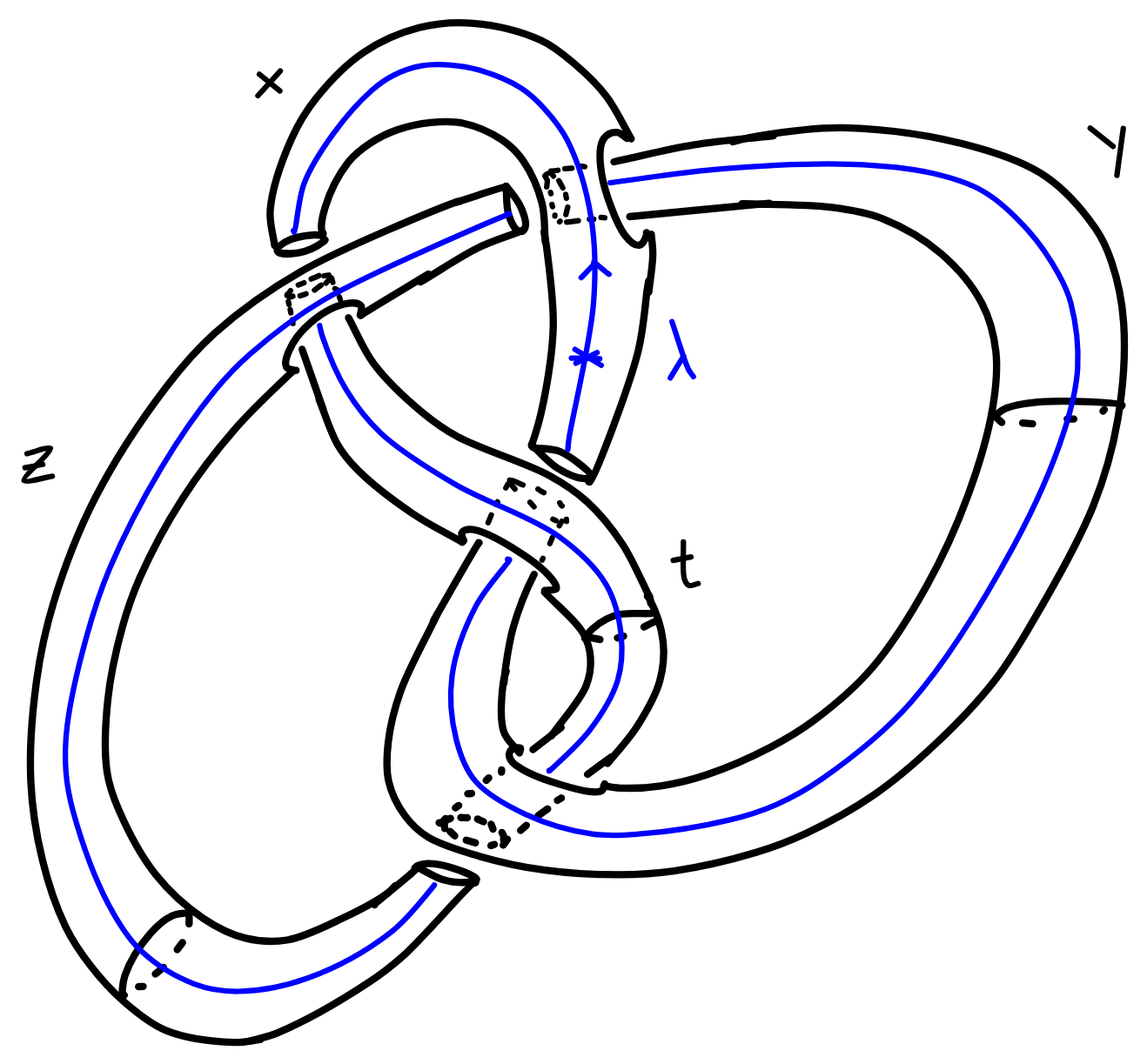}

Figure 1: The spun torus of the figure 8 knot, represented as a tube \cite{satohvknot}, with longitude $\lambda=z^{-1}yx^{-1}tx^0$.
\end{center}

 This method of calculation and the following proof were inspired by the quandle 2-cocycle invariant of \cite{rcquandles}. Indeed, the proof below is very similar to the proof of Theorem 1.2 in \cite{rcquandles}, and prompts the subsequent question. In the quandle 2-cocycle calculation, one chooses a curve $\lambda$ representing a homology class on a diagram of a surface knot, then computes a quandle cocycle calculation with respect to a fixed 2-cocycle $\theta$ at each double curve undercrossing. 

\begin{proof}[Proof of Lemma \ref{lem:periph}]
Let $D_0$ be a diagram for $K_0$. Then a diagram $D_1$ for $K_1$ is obtained from $D_0$ by taking a split union of $D_0$ with an unlink of 2-spheres, and joining them along the boundaries of some 3-dimensional 1-handles which are allowed to link the rest of the diagram. The solid 1-handles intersect the rest of the surface in disks, which can be assumed to be as small as we like, hence miss any double curve crossings. 

Now, we can choose a generating system of curves for $\pi_1$ of the surface on $D_0$ which miss all of the double curve crossings from the 1-handles to be joined. Call this system of curves $\{\gamma_1,\dots,\gamma_{2g}\}$. Since this is a ribbon concordance, the same system of curves generates $\pi_1$ of the surface on $D_1$. Moreover, the image of $\gamma_i$ under the homomorphism $\phi$ is the `same' curve $\gamma_i$ considered on $D_1$. Therefore every $\gamma_i$ on $D_1$ is in the image of $\phi|_{P(K_0)}$. If $x$ is the meridian for the basepoint region on $K_0$, then $\phi(x)$ is the corresponding meridian for $K_1$, so $\phi(P(K_0))=P(K_1)$. 
\end{proof}

\begin{question}
Let $\lambda$ be a curve on a diagram for a surface knot. If the 2-cocycle invariant $\theta_\lambda(K)$ is nontrivial, must it be the case that $\lambda\neq1\in P(K)$? 
\end{question}

If the answer to this question is no, then the quandle 2-cocycle invariant could theoretically detect irreducibility of knotted surfaces when the peripheral subgroup fails (see Example 3 below). For instance, the standard way to show a knotted torus is not a 2-knot with a trivial handle attached is to show that the peripheral subgroup is bigger than $\mathbb{Z}$. It would be remarkable if the quandle invariant could be nontrivial even when the peripheral subgroup is infinite cyclic. On the other hand, if the answer is yes then this would provide an interesting link between quandle cocycle invariants and the peripheral subgroup.

Now, if $K_0\rightarrow K_1$ is any ribbon concordance and $\pi K_0$ is residually finite or locally indicable, then $\phi$ is injective, as pointed out in Remark \ref{rem:inj}, so by the previous lemma $P(K_0)\cong P(K_1)$.

\peripheral*

\begin{proof}
If $K_0$ is 0-concordant to $K_1$, then there exists a surface knot $J$ and ribbon concordances $K_0\rightarrow J\leftarrow K_1$. Each of the induced homomorphisms \newline $\pi K_0\rightarrow \pi J\leftarrow \pi K_1$ is injective, so $P(K_0)\cong P(J)\cong P(K_1)$.
\end{proof}

Since 0-concordance doesn't involve the genus of a knotted surface in any significant way, this is not surprising. We conjecture that the hypothesis on the groups may be removed. 

\begin{example}
First we note some examples of surface knots whose knot groups are residually finite. The twist spun torus knots and turned twist spun torus knots of Boyle \cite{boyleturned} have the same groups as the twist spun 2-knots of Zeeman. Since these 2-knots are fibered \cite{zeeman}, their groups are residually finite, so this gives a large number of knotted tori for which the peripheral subgroup is a 0-concordance invariant. Moreover, the property of being residually finite for knot groups is closed under taking connected sums: if $\pi K$ and $\pi J$ are residually finite, then $\pi (K\# J)$ is residually finite as well \cite{bolerevans}.

\end{example}

In the next examples we show that peripheral subgroups and ideal classes each can distinguish some 0-concordance classes when the other cannot. We also point out that there are many ribbon torus knots which are not 0-concordant to the unknot, in contrast with the case of 2-knots.

\begin{example}

    1. Let $K$ be a nontrivial classical knot. Then the spun torus knot of $K$ is a knotted torus with peripheral subgroup $\mathbb{Z}^2$ (for $K=4_1$, see Figure 1, pictured as the ``tube" of $K$). This is a ribbon torus knot which is not 0-concordant to the unknotted torus. Since it has a classical knot group, its Alexander ideal is principal.\\
    2. Let $K$ be a 2-bridge knot with nonunit determinant and let $J$ be the 2-twist spun torus knot of $K$. Boyle proved that these tori are reducible, so $P(J)\cong\mathbb{Z}$, however $\Delta(J)$ is nonprincipal by Corollary \ref{cor:det}, so $J$ is not 0-slice because its ideal class is nontrivial (more generally, attach a trivial handle to any 2-knot with nonprincipal ideal).\\
    3. The twist spun torus knots of the previous example can all be replaced with ribbon torus knots, by starting with the spun 2-knot of $K$ and attaching a handle to effect the relation $y^{-2}xy^2=x$. One can check that the longitude of this handle is trivial.   
     \end{example}

\begin{remark}The ribbon tori of Example 3 are most likely irreducible, although this is difficult to prove. A ribbon torus has peripheral subgroup at most $\mathbb{Z}^2$, but in \cite{litherlandH2} there are examples of tori with peripheral subgroup $\mathbb{Z}^3$. It seems likely that these tori are not 0-concordant to any ribbon torus. 
\end{remark}


\end{section}

\begin{section}{Questions}
It is interesting to consider the directed graph of ribbon concordances. The vertices are surface knots, and there is a directed edge from $K_0$ to $K_1$ whenever there is a ribbon concordance $K_0\rightarrow K_1$. Two surface knots are 0-concordant if and only if they are in the same connected component of the ribbon concordance graph. We say $K$ is a {\it root} of the graph if whenever $J\rightarrow K$, $J\cong K$ (such a $K$ is called {\it minimal} in \cite{gordonrc}). A chain of ribbon concordances $\dots\rightarrow K_{-1}\rightarrow K_0\rightarrow K_1$ gives rise to an ascending chain of ideals $\Delta(K_1)\subseteq\Delta(K_0)\subseteq\Delta(K_{-1})\subseteq\cdots$ by the key lemma. Since $\mathbb{Z}[t^{\pm1}]$ is Noetherian, this chain must stabilize. This is some suggestion that the ribbon concordance graph has roots. Question \ref{q:effectiveness} is related to the following question:

\begin{question}
Does any connected component of the ribbon concordance graph have more than one root?    
\end{question}

For instance, if $\Delta(K)$ is a maximal ideal, then by the key lemma it must be true that for any ribbon concordance $J\rightarrow K$, $\Delta(J)=\Delta(K)$. Of course this does not imply that $J\cong K$, but under some additional assumption this is perhaps the case, e.g.\ when $\pi K$ is prime. A knot group $\pi K$ is {\it prime} if $\pi K=\pi (J\#J^\prime)$ implies $\pi J$ or $\pi J^\prime$ is infinite cyclic, e.g.\ if its commutator subgroup is indecomposable as a free product. One such family of examples is the 2-twist spins of $(2,p)$-torus knots, $K_p$, for $p$ an odd prime. In this case the knot group has a simple decomposition: $\pi K_p\cong\langle x,a|xax^{-1}=a^{-1},a^p=1\rangle\cong\mathbb{Z}_p\rtimes\mathbb{Z}$, with commutator subgroup $\mathbb{Z}_p$. 

\begin{proposition}
Let $J\rightarrow K_p$ be a ribbon concordance. If the induced homomorphism $\phi:\pi J\rightarrow \pi K_p$ is injective, then it is an isomorphism preserving meridians.
\end{proposition}

\begin{proof}

 By Proposition \ref{prop:hom}, the induced map $\phi$ always takes meridians to meridians. Now, $\phi$ restricts to an injection of commutator subgroups, so $(\pi J)^\prime\cong 1$ or $\mathbb{Z}_p$. It can't be $1$, since then $\pi J\cong\mathbb{Z}$, so $\Delta(K_p)=(f)\Delta(J)=(f)$ would be principal. So $(\pi J)^\prime\cong\mathbb{Z}_p$, and $\phi$ maps this isomorphically onto $(\pi K_p)^\prime$. Therefore the image of $\phi$ contains $\mathbb{Z}_p$ and a meridian, but this is enough to generate $\pi K_p$, so $\phi$ is surjective as well as injective.

\end{proof}

Note that since the group map $\phi$ is injective and preserves meridians, the induced quandle map $\varphi$ is also an isomorphism. By Theorem 1.1 of \cite{rcquandles}, $J$ would also have a nontrivial output with the quandle 3-cocycle invariant. It seems likely this would necessarily be the same output as for $K_p$; this is the case when $p=3$, since the tricoloring quandle $R_3$ is triply symmetric and so the output under each nontrivial coloring is the same.

\begin{conjecture}
For $p$ prime, the 2-twist spin of $T(2,p)$ is a root of the ribbon concordance graph.
\end{conjecture}

Modulo the unknotting conjecture, it seems likely that the unknots $\mathcal{U}_g$ are roots as well. If there exists a ribbon concordance $K\rightarrow \mathcal{U}_g$, then $\pi K\cong\mathbb{Z}$ or is not residually finite (see Remark \ref{rem:inj}).

As pointed out in the introduction, 0-concordance is the smallest equivalence relation generated by ribbon concordance, which parallels the case of classical knot concordance precisely. Therefore, both cases have a natural slice-ribbon problem. Cochran produced nonribbon 2-knots in \cite{cochran} which are 0-null-bordant (allowing for 3-manifolds with 2 boundary components besides $S^2\times I$), but as far as we know there are no examples of 0-slice, nonribbon 2-knots. Such a 2-knot $K$ would have a ribbon concordance $K\rightarrow J$, where $J$ is ribbon.

\begin{question}
Is every 0-slice 2-knot ribbon?
\end{question}

The techniques of this paper show that if a 2-knot $K$ is invertible in $\mathscr{K}_0$, then $\Delta(K)$ must be principal.

\begin{question}
Is any nontrivial 0-concordance class invertible?
\end{question}

\end{section}




\begin{thebibliography}{Kaw96}

\bibitem[BE73]{bolerevans} 
James Boler and Benny Evans. 
\newblock The free product of residually finite groups amalgamated along retracts is residually finite.
\newblock {\em Proc. Amer. Math. Soc.} 37 (1973), pp. 50–52.

\bibitem[Boy93]{boyleturned} 
Jeffrey Boyle. 
\newblock The turned torus knot in $S^4$.
{\em J. Knot Theory Ramifications} 2.3 (1993), pp. 239–249.

\bibitem[CSS06]{rcquandles} 
J. Scott Carter, Masahico Saito, and Shin Satoh. 
\newblock Ribbon concordance of surface-knots via quandle cocycle invariants. 
{\em J. Aust. Math. Soc.} 80.1 (2006), pp. 131–147. 

\bibitem[Coc83]{cochran} 
Tim Cochran. 
\newblock Ribbon knots in $S^4$. 
\newblock {\em J. London Math. Soc.} (2) 28.3 (1983), pp. 563–576. 

\bibitem[CF63]{crowfox} 
Richard H. Crowell and Ralph H. Fox. 
\newblock {\em Introduction to knot theory.} 
\newblock Ginn and Co., Boston, Mass. 1963, pp. X+182.

\bibitem[DM19]{daimiller} 
Irving Dai and Maggie Miller. 
\newblock The 0-concordance monoid is infinitely generated. 
\newblock 2019. arXiv: 1907.07166 [math.GT].

\bibitem[Eis95]{eisenbud} 
David Eisenbud. 
\newblock {\em Commutative Algebra with a View Toward Algebraic Geometry.}
\newblock Vol. 150. Graduate Texts in Mathematics. Springer-Verlag, New York, 1995.

\bibitem[Fox53]{foxcalci} 
Ralph H. Fox. 
\newblock Free differential calculus I. Derivation in the free group ring. 
\newblock {\em Ann. of Math.} (2) 57 (1953), pp. 547–560.

\bibitem[Fox62]{foxqt} 
R. H. Fox. 
\newblock  A quick trip through knot theory.
\newblock {\em Topology of 3-manifolds and related topics} (Proc. The Univ. of Georgia Institute, 1961). PrenticeHall, Englewood Cliffs, N.J., 1962, pp. 120–167. 

\bibitem[Gor81]{gordonrc} 
C. McA. Gordon. 
\newblock Ribbon concordance of knots in the 3-sphere.  
\newblock {\em Math. Ann.} 257.2 (1981), pp. 157–170.

\bibitem[Gut72]{gutierrez} 
M. A. Guti\'{e}rrez. 
\newblock On knot modules.  
\newblock {\em Invent. Math.} 17 (1972), pp. 329– 335. 

\bibitem[Kan83]{kanenobusat} 
Taizo Kanenobu. 
\newblock Groups of higher-dimensional satellite knots. 
\newblock {\em J. Pure Appl. Algebra} 28.2 (1983), pp. 179–188. 

\bibitem[KK94]{kankaz} 
Taizo Kanenobu and Ken-ichiro Kazama. 
\newblock The peripheral subgroup and the second homology of the group of a knotted torus in $S^4$. 
\newblock {\em Osaka J. Math.} 31.4 (1994), pp. 907–921. 

\bibitem[Kaw96]{kawauchi} 
Akio Kawauchi. 
\newblock {\em A survey of knot theory.} Translated and revised from the 1990 Japanese original by the author. 
\newblock Birkh\"{a}user Verlag, Basel, 1996, pp. xxii+420. 

\bibitem[Ker65]{kervaire} 
Michel A. Kervaire. 
\newblock Les nœuds de dimensions sup\'{e}rieures. 
\newblock {\em Bull. Soc. Math. France} 93 (1965), pp. 225–271. 

\bibitem[Kin61]{kinoshita} 
Shin’ichi Kinoshita. 
\newblock On the Alexander polynomials of 2-spheres in a 4-sphere.  
\newblock {\em Ann. of Math.} (2) 74 (1961), pp. 518–531. 

\bibitem[Kir97]{kirby} 
Problems in low-dimensional topology.
\newblock {\em Geometric topology} (Athens, GA, 1993). Ed. by Rob Kirby. Vol. 2. AMS/IP Stud. Adv. Math. Amer. Math. Soc., Providence, RI, 1997, pp. 35–473. 

\bibitem[Lev78]{levine} 
J. Levine. 
\newblock Some results on higher dimensional knot groups.  
\newblock {\em Knot theory (Proc. Sem., Plans-sur-Bex, 1977).} Vol. 685. Lecture Notes in Math. With an appendix by Claude Weber. Springer, Berlin, 1978, pp. 243– 273. 

\bibitem[Lit81]{litherlandH2} 
R. A. Litherland. 
\newblock The second homology of the group of a knotted surface.
\newblock {\em Quart. J. Math. Oxford Ser.} (2) 32.128 (1981), pp. 425–434. 

\bibitem[Mat86]{matsumura} 
Hideyuki Matsumura. 
\newblock {\em Commutative ring theory.}
\newblock Vol. 8. Cambridge Studies in Advanced Mathematics. Translated from the Japanese by M. Reid. Cambridge University Press, Cambridge, 1986, pp. xiv+320. 

\bibitem[Mel77]{melvin} 
Paul Michael Melvin. 
\newblock {\em Blowing up and down in 4-manifolds.}
\newblock Thesis – University of California, Berkeley. 1977. 

\bibitem[Sat00]{satohvknot} 
Shin Satoh. 
\newblock Virtual knot presentation of ribbon torus-knots. 
\newblock {\em J. Knot Theory Ramifications} 9.4 (2000), pp. 531–542. 

\bibitem[Sun15]{sunukjianconc} 
Nathan S. Sunukjian. 
\newblock Surfaces in 4-manifolds: concordance, isotopy, and surgery.  
\newblock {\em Int. Math. Res. Not.} IMRN 17 (2015), pp. 7950–7978. 

\bibitem[Sun19]{sunukjian0conc} 
Nathan Sunukjian. 
\newblock 0-Concordance of 2-knots. 
\newblock 2019. arXiv: 1907.06524 [math.GT]. 

\bibitem[Zee65]{zeeman} 
E. C. Zeeman. 
\newblock Twisting spun knots.
\newblock {\em Trans. Amer. Math. Soc.} 115 (1965), pp. 471–495.
      

\end{thebibliography}
\end{document}